\documentclass[a4paper,12pt]{amsart}
\usepackage{url}
\usepackage{latexsym}
\usepackage{lscape} 
\usepackage[top=30truemm,bottom=35truemm,left=22truemm,right=22truemm]{geometry}
\usepackage{pdflscape}
\usepackage{float}
\usepackage{amssymb}
\usepackage{amscd}
\usepackage{graphicx}
\usepackage[normalem]{ulem}
\usepackage{enumitem}
\usepackage{comment}
\usepackage{tikz, tikz-cd}
\usetikzlibrary{decorations.markings}
\tikzset{
  midarrow/.style={
    postaction={decorate},
    decoration={markings, mark=at position 0.5 with {\arrow{stealth}}}
  }
}
\usepackage[all]{xy}
\usepackage{subfiles} 
\usepackage[extrasp=0em,mono]{inconsolata}  
\usepackage[backend=biber,  
style=alphabetic,
  maxbibnames=99,
  maxcitenames=6,
sorting=nyt, 
url=false, 
isbn=false,
eprint=true,
doi=false,
]{biblatex}
\AtEveryBibitem{%
  \clearlist{language}%
  \clearfield{edition}%
  \clearname{translator}%
  \clearfield{number}%
}
\renewbibmacro{in:}{}
\DeclareFieldFormat
  [article,incollection,inproceedings]
  {title}{#1}

\usepackage{soul,cancel}

\newcommand{\skipover}[1]{}

\newtheorem{thm}[equation]{Theorem}
\newtheorem{lem}[equation]{Lemma}
\newtheorem{cor}[equation]{Corollary}

\theoremstyle{remark}
\newtheorem{rem}[equation]{Remark}

\theoremstyle{definition}

\numberwithin{equation}{section}

\newcommand{\R}{\mathbb{R}}
\newcommand{\C}{\mathbb{C}}

\newcommand{\rank}{\operatorname{rank}}
\newcommand{\End}{\operatorname{End}}

\newcommand{\Ad}{\operatorname{Ad}}
\newcommand{\ad}{\operatorname{ad}}

\newcommand{\lie}[1]{\mathfrak{#1}}
\newcommand{\Lie}{\operatorname{Lie}}

\newcommand{\half}{\frac{1}{2}}

\newcommand{\be}{\begin{equation}}
\newcommand{\beu}{\begin{equation*}}

\newcommand\Cl {\mathrm{Cl}} 
\newcommand\ioa{\ioa}

\newcommand{\bbar}{\,\big|\,}

\usepackage{graphicx,color,cancel}

\newcommand\la{{\lambda}}

\newcommand\twedge{\textstyle{\bigwedge}}

\newcommand\Sp{\operatorname{Sp}}
\newcommand\USp{\operatorname{USp}}

\newcommand\stab{\operatorname{Stab}}

\newcommand\GL{\operatorname{GL}}
\newcommand\SO{\operatorname{SO}}
\newcommand\OO{\operatorname{O}}
\newcommand\UU{\operatorname{U}}

\newcommand\Mat{\operatorname{Mat}}
\newcommand\Sym{\operatorname{Sym}}
\newcommand\Alt{\operatorname{Alt}}
\newcommand\Her{\operatorname{Her}}

\newcommand\Gr{\operatorname{Gr}}
\newcommand\im{\operatorname{im}}
\newcommand\Ker{\operatorname{ker}}

\newcommand\Cas{\operatorname{\Omega}}

\newcommand{\pf}{\begin{proof}}
\newcommand{\epf}{\end{proof}}
\newcommand{\eq}{\begin{equation}}
\newcommand{\eeq}{\end{equation}}
\newcommand{\eqn}{\begin{equation*}}
\newcommand{\eeqn}{\end{equation*}}

\newcommand{\fra}{\mathfrak{a}}
\newcommand{\frb}{\mathfrak{b}}

\newcommand{\frg}{\mathfrak{g}}

\newcommand{\frk}{\mathfrak{k}}
\newcommand{\frl}{\mathfrak{l}}
\renewcommand{\frm}{\mathfrak{m}}
\newcommand{\frn}{\mathfrak{n}}

\newcommand{\frp}{\mathfrak{p}}
\newcommand{\frq}{\mathfrak{q}}
\newcommand{\frr}{\mathfrak{r}}
\newcommand{\frs}{\mathfrak{s}}
\newcommand{\frt}{\mathfrak{t}}

\newcommand{\bbC}{\mathbb{C}}

\newcommand{\bbH}{\mathbb{H}}

\newcommand{\HLGr}{\operatorname{HLGr}}

\newcommand{\smat}{\left(\begin{smallmatrix}}
\newcommand{\esmat}{\end{smallmatrix}\right)}
\newcommand{\pmat}{\begin{pmatrix}}
\newcommand{\epmat}{\end{pmatrix}}

\newcommand{\bbP}{\mathbb{P}}

\newcommand{\frP}{\mathfrak{P}}

\newcommand{\transpose}[1]{\,{}^t{#1}}

\newcommand{\wzerol}{w_0^{\frl}}
\newcommand{\wzerog}{w_0^{\frg}}

\title[Symmetric spaces and skew multiplicity free representations]{Symmetric spaces and skew multiplicity free representations}
\newcommand{\version}{Ver.~0.0}
\newcommand{\setversion}[1]{\renewcommand{\version}{Ver.~{#1}}}
\setversion{0.01 [2025/11/17, 13:43]}
\setversion{0.02 [2026/03/01 15:30:10 JST]}
\setversion{0.10503 [2026/05/03 23:59:45 JST]}
\setversion{0.260505 [2026/05/05 13:59:06 JST]}
\setversion{0.260809 [2026/08/09 15:31:58 JST]}
\setversion{0.3 [2026/09/04 13:55:11 JST]}
\setversion{0.4 [2026/09/11 14:57:33 JST]}
\setversion{0.5 [2026/09/12 13:55:11 JST]}


\author{Kieran Calvert}
\address[Calvert]{School of Mathematical Sciences, Charles Carter Building, Lancaster University, Lancaster, UK}
\email{kieran.calvert@lancaster.ac.uk}
\author{Kyo Nishiyama}
\address[Nishiyama]{Department of Mathematics, Aoyama Gakuin University, Fuchinobe 5-10-1, Chuo-ku,
Sagamihara 252-5258, Japan}
\email{kyo.nishiyama@gmail.com}
\author{Pavle Pand\v zi\'c}
\address[Pand\v zi\'c]{Department of Mathematics, Faculty of Science, University of Zagreb, Bijeni\v cka 30, 10000 Zagreb, Croatia}
\email{pandzic@math.hr}
\thanks{K.~Calvert is supported by JSPS fellowship programme PE25758}
\thanks{K.~Nishiyama is supported by JSPS KAKENHI Grant Number \#{25K06938}.}
\thanks{P.~Pand\v{z}i\'{c} is supported by "Implementation of cutting-edge research and its application as part of the Scientific Center of Excellence for Quantum and Complex Systems, and Representations of Lie Algebras", Grant No. PK.1.1.10.0004, co-financed by the European Union through the European Regional Development Fund - Competitiveness and Cohesion Programme 2021- 2027., by the Croatian Science Foundation (HRZZ), grant no. IP-2025-02-6514, and by the European Union – NextGenerationEU through the National Recovery and Resilience Plan 2021-2026. Institutional grant of University of Zagreb Faculty of Science (IK IA 1.1.3. Impact4Math).}
\date{Current \version, compiled on \today}

\newcommand{\itG}{G}
\newcommand{\itK}{K}
\newcommand{\itP}{P}
\newcommand{\itA}{A}
\newcommand{\itB}{B}
\newcommand{\itM}{M}
\newcommand{\itL}{L}
\newcommand{\itN}{N}

\begin{document}

\maketitle

\begin{abstract}
Let $ G $ be a non-compact connected reductive real Lie group 
and $ P $ its parabolic subgroup.   
We give a list of equivalent conditions for $ P $ to have an abelian unipotent radical.  
These conditions are related to skew multiplicity free and multiplicity free actions, 
and using the theory of Clifford algebra and the cubic Dirac operator we deduce formulas for 
the de Rham cohomology of compact symmetric spaces $ G/P $.
\end{abstract}

\section{Introduction}

Let $ G $ be a non-compact connected reductive real Lie group, with a maximal compact subgroup $ K $, and 
denote by $ G = K A N_0 $ the Iwasawa decomposition, where $ A $ is a maximally split torus and $ N_0 $ the associated maximal unipotent subgroup.  Let us fix a minimal parabolic subgroup $ P_0 = M A N_0 $, where $ M = Z_K(A) $, the centralizer of $ A $ in $ K $.  
We denote the Cartan involution fixing $ K $ by $\theta$.  
The Lie algebra of $ G $ is denoted by the corresponding German small letter $ \lie{g} $ (for other subgroups we follow the same notational convention) and we denote the differential of $ \theta $ again by $\theta$. Thus $ \theta $ is an involution on $ \frg $.  

In this paper, we consider a parabolic subgroup $ P $ of $ G $ with unipotent radical $ N $.  
Without loss of generalities, we can assume $ P $ contains $ P_0 $ so that $ N \subset N_0 $.  
We denote the standard Levi subgroup of $ P $ by $ L $ so that $ P = L N $ gives a Levi decomposition.
$ L $ naturally acts on $ \frn $, the nilpotent radical of the Lie algebra of $ P $, by the adjoint action.

If the unipotent radical $ N $ of $ P $ is abelian (or equivalently to say, the nilpotent radical $ \frn $ is abelian), 
so many miracles occur, which are listed in Theorem \ref{thm:SSP.SMF.MF.theorem} and reproduced briefly here (alphabetical indicators of the items are synchronized with those in the theorem):
    \begin{enumerate}[label={\makebox[3ex][c]{\upshape(\alph*)}}]
\setcounter{enumi}{1}
        \item $\itG/\itL$ is symmetric. \label{intropointb}
        \item Define $ R = K \cap P $, 
then $ R $ is a symmetric subgroup of $ K $, i.e., $ \itK/R $ is a compact symmetric space. \label{intropointc}
        \item $\frn_\C = \frn\otimes_\R \C$ is a skew multiplicity free $\itL_\C$ module under the adjoint action. 
In other words, $ \twedge(\frn_\C) $ is a multiplicity free $ L_\C $-module.  \label{intropointd}
        \item The dimension of $\twedge(\frn_\C \oplus \theta \frn_\C)^{\itL_\C}$ is equal to $|W_{\frg_\C}|/|W_{\frl_\C}|$, 
where $ W_{\frg_\C} $ and $ W_{\frl_\C} $ are Weyl groups. \label{intropointe}
        \item The $ \itL_\C $ module $\twedge(\frn_\C)$ has $|W_{\frg_\C}|/|W_{\frl_\C}|$ irreducible submodules. \label{intropointf}
        \item The Casimir $\Cas_{\frl_\C}$ of $ \frl_\C $ acts by scalar 
on the exterior algebra $\twedge(\frn_\C \oplus \theta \frn_\C)$.\label{intropointg}
        \item The $\frn$ cohomology has zero differential.\label{intropointh}
        \item $ \itG_\C/ \itP_\C $ is a compact Hermitian symmetric space.\label{intropoint.i}
\item The $ \itL_\C $-module $ \frn_\C $ is multiplicity free (or spherical), i.e., the polynomial ring $ \C[\frn_\C] $ is a multiplicity free $ L_\C $-module. \label{intropointj}
\item $ \itG_\C/ \itP_\C $ is an $ \itL_\C $-spherical variety.\label{intropointk}
\item If $ \itB_{\itL_\C} $ denotes a Borel subgroup of $ L_\C $, 
the double flag variety $ \itL_\C/ \itB_{\itL_\C} \times \itG_\C/ \itP_\C $ is of finite type.\label{intropointl}
    \end{enumerate}
Moreover, the conditions \ref{intropointb} -- \ref{intropoint.i} are equivalent and they imply that $ \frn $ is abelian.  

From the above claims, if the unipotent radical is abelian, it is transparent that $ P $ is a maximal parabolic subgroup and hence 
$ G/P $ is a real Grassmannian manifold (see \cite{CNP.arxiv2023}, for example).  

Some remarks are in order (see Remark \ref{remark:to.theorem.P.with.abelian.nilradical} after the theorem for more details).  
Relation to the symmetric spaces \ref{intropointb} and \ref{intropointc} are well known, which appears in \cite[\S{8.2}]{Kostant.I.1961} and \cite{TK1968}, for example.  See also \cite{Nagano.1965} and \cite{Kobayashi.Nagano.1964,Kobayashi.Nagano.1965}.  
The condition that the nilradical $ \frn_\C $ is a skew multiplicity free $ L_\C $-module was obtained 
by Kostant \cite[Corollary 8.4]{Kostant.I.1961} in the case where $ G $ is complex reductive (and also \ref{intropoint.i} is obtained in \S 8.2, \textit{ibid}.), 
and this fact is also used extensively by Howe to classify skew multiplicity free actions \cite{Howe.SchurLecture.1995}.  
The conditions \ref{intropointe} -- \ref{intropointg} are related to the theory of Clifford algebras and cubic Dirac operators \cite{Kostant.cubic.1999,Goette.MZ.1999}. The statement that $\frn$ being abelian implies \ref{intropointh} was proved by Kostant \cite{Kostant.I.1961}. The conditions \ref{intropoint.i} and \ref{intropointj} are studied by Richardson-Rohrle-Steinberg in the case of algebraic group $ G $ over an algebraically closed field. The statements \ref{intropointj} and \ref{intropointl} are a rephrase of \ref{intropointk} in a sense. For this, see \cite{Fresse.Nishiyama.Overview}.

Thus many of these properties are known and studied in many places separately in a somewhat ad hoc manner.  
Our treatment of these properties related to $ P $ with abelian unipotent radical is uniform and general.  
The proofs are natural and simple, not depending on case-by-case analysis.  
So in a sense, Theorem \ref{thm:SSP.SMF.MF.theorem} is one of our main results in this paper.

However, there is a lot more to say about the Grassmannian $ G/P $, where $ P $ has abelian unipotent radical.  
We discuss the de Rham cohomology of $ G/P $ and describe it in terms of the action of the cubic Dirac operator and 
the $ L_\C $-module structure on $ \twedge(\frn_\C) $.  
The ring structure naturally comes from the exterior product of invariants of $ R = K \cap P = K \cap L $.  
Since $ R $ is a symmetric subgroup of $ K $ and $ G/P \simeq K/R $, 
$ R $ invariants of the multiplicity free representation $ \twedge(\frn_\C) $ of $ L_\C $ can be described in a beautiful manner.  Cartan and Borel gave a uniform description of the de Rham cohomology of a homogeneous space $K/R$ if $K$ and $R$ are both compact and connected and the pair is Cartan, which symmetric spaces always are  \cite{Cart50,Borel.1953}. In  Theorem \ref{thm:CartanBorelforDisconnected} we extend this description of the de Rham cohomology of symmetric pairs $K/R$ to the case when only $K$ is required to be compact and connected, while $R$ can be an arbitrary, possibly disconnected, closed subgroup.  
We make these results precise in \S \ref{sec:deRham.cohomology} and \S \ref{sec:CartanBorelforDisconnected}, and give many interesting examples in \S \ref{section:examples}.
Although it seems the explicit results by calculation are known to experts, 
our emphasis here is the uniform description of the cohomology ring using only representation theory of $ L $ and $ R $.  
We believe that our proofs of these results are of interest because they are accessible to anyone with knowledge of basic representation theory.

\section{Preliminaries}

Let $\itG$ be a non-compact connected reductive real Lie group, with a maximal compact subgroup $\itK$ and associated Cartan involution $\theta$. Let $\frg = \Lie(\itG)$ be the Lie algebra of $\itG$. We denote the differential of the Cartan involution by the same letter, so that $\theta$ is an involution on $\frg$, such that $\Lie(\itK) = \frk = \frg^\theta$, $\frp = \frg^{-\theta}$ and $\frg = \frk \oplus \frp$. 
Let $\fra$ be a maximal abelian subalgebra of $\frp$ and let $\Delta =\Delta(\frg,\fra)$ be the restricted roots of the action of $\fra$ on $\frg$.   
We fix a set of positive roots $ \Delta^+ \subset \Delta$, 
and consequently the simple system of roots $\Pi \subset \Delta^+$.  
Define $\frn_0 = \sum_{\alpha \in \Delta^+}\frg_\alpha$, we arrive at the Iwasawa decompositions for $\frg$ and $\itG$;
\[ \frg = \frk \oplus \fra \oplus \frn_0, \qquad \itG = \itK \itA \itN_0,\]
with $\itA = \exp(\fra), \itN_0= \exp(\frn_0)$. Let $\frm = \{x \in \frk: [x,\fra] = 0\} $ 
and $ \itM = Z_{\itK}(\fra) $, which is not connected in general.   
The restricted root decomposition is
\[
\frg = \frm \oplus \fra \oplus \bigoplus_{\alpha \in \Delta} \frg_\alpha.  
\]
For a fixed choice of positive roots, $\itG$ has a unique minimal parabolic subgroup containing $\itA\itN_0$, denoted by $\itP_0 = \itM \itA \itN_0$. Every parabolic subgroup is conjugate to a unique parabolic subgroup $\itP$ containing $\itP_0$ and we have the following construction of such parabolic subgroups.

\begin{lem}
    There is a bijection between the following sets 
    \[ \{ \itP: \itP \supset \itP_0\} \leftrightarrow \{ J \subset \Pi\} .\]
\end{lem}

\begin{proof}
This lemma is fairly well known, but since we need the construction below, we give a brief account.  

Consider $J \subset \Pi$ and choose $h \in \fra$ such that
\[\alpha_j(h) = \begin{cases} 
0 & \alpha_j \in J, \\
>0 & \alpha_j \in \Pi\setminus J.
\end{cases}\]
Define $\frl = \{ x \in \frg : [h,x] = 0\}$. 
Then $\frl$ contains a Cartan subalgebra of $\frg$, hence $ \frl $ is its own normalizer, 
which implies the analytic subgroup generated by $ \exp(\frl)$ is closed, hence a Lie subgroup of $\itG$.  
In fact, the normalizer of a subspace in $ \frg $ is defined by polynomial equations so it is closed (see  \cite[Remark 5.1]{Kostant.I.1961}).
We denote this subgroup by $\itL = Z_{\itG}(h) $, which contains $ M $ as a subgroup by definition.  Note that $\itL$ again need not be connected, and its connected component is the analytic subgroup corresponding to $ \frl $.  
Let $\frn = \bigoplus_{\substack{
\alpha(h) >0}} \frg_\alpha$ which is a nilpotent Lie subalgebra of $\frg$, and define $\itN = \exp(\frn)$.  Since 
\begin{equation*}
\frl \oplus \frn = \frm \oplus \fra \oplus \frn_0 \oplus \bigoplus_{\alpha \in \Delta^-:\; \alpha(h) = 0} \frg_{\alpha} 
\end{equation*}
then it follows that $\itP = \itL \itN$ contains $\itP_0$.   
We denote the Lie algebra of $ \itP $ by $ \frP $ since we use $ \frp $ for the Cartan space orthogonal to $ \frk $.  

For the reverse direction, let us consider the Lie algebra of $ \itP $.  Since it contains $ \fra $, it has a root space decomposition.  
We can pick $ h \in \fra $ satisfying 
(a)\ $ \alpha(h) \geq 0 $ for any $ \alpha \in \Delta^+ $; and 
(b)\ if $ \alpha $ is a root of $ \frP $ and $ \alpha(h) > 0 $, then $ -\alpha $ is not a root of $ \frP $.  
For this $ h $, we can recover $ \frP = \frl \oplus \frn $ exactly in the same way as above.  
This element $ h $ defines $ J \subset \Pi $.
%
\end{proof}

Let $\itP$ be a parabolic subgroup of $\itG$ containing $\itP_0$.  Using the notation in the proof of the above lemma, 
then $\Lie(\itP) =\frP= \frl \oplus \frn$.  
Note that $ \theta(h) = - h $ since $ h \in \frg^{-\theta} $, and this implies $ \theta(\frg_{\alpha}) = \frg_{-\alpha} $ for $ \alpha \in \Delta $.  Therefore we conclude that 
\begin{equation*}
\theta\frn = \bigoplus_{\alpha \in \Delta^+:\; \alpha(h) >0} \theta(\frg_{\alpha}) 
=  \bigoplus_{\alpha \in \Delta^-:\; \alpha(h) <0} \frg_{\alpha}  
\end{equation*}
and arrive at the following decomposition of $\frg$:
\[\frg = \frl \oplus \frn \oplus \theta \frn.\]
Since $\fra \subset \frg^{-\theta}$ and $\frm \subset \frk = \frg^{\theta} $, then $\theta|_{\fra} = -1$ and $\theta|_\frm = 1$, furthermore
\begin{equation}\label{eq:description.of.lie.L}
\frl = \frm \oplus \fra \oplus \bigoplus_{\alpha:\; \alpha(h)=0} \frg_\alpha
\end{equation}
is $\theta$ stable. Let $\frs = \frn \oplus \theta \frn$. 

\begin{lem}
Let $\itP$ be a parabolic subgroup of $\itG$ and put $R = \itK \cap \itP$.   
Then there is an isomorphism of homogeneous spaces 
\[\itG/\itP \cong \itK/R,\]
and $R = \itL^\theta$, i.e., $ R $ is a maximal compact subgroup of $ L $. (Note that $ L $ is not connected in general.)
\end{lem}

\begin{proof}
Without loss of generality, assume that $\itP \supset \itP_0$ (otherwise take a conjugate of $\itP$).  
Note that $\itG/\itP_0$ is $\itK\itA\itN_0/\itM\itA\itN_0$ and hence is a single $\itK$ orbit. Furthermore the surjection $\itG/\itP_0 \to \itG/\itP$ is $\itG$ equivariant, thus $\itG/\itP$ is $\itK$ homogeneous. Hence $\itG /\itP \cong \itK /\itK \cap \itP = \itK/R$. The group $\itN$ is nilpotent, thus every nontrivial element generates a non compact group. $\itK$ is compact and so $\itK \cap \itN = \{ e \}$ ($ e $ denotes the unit), hence $\itK \cap \itP = \itK \cap \itL = \itL^\theta$.  
\end{proof}

\section{Parabolic subgroups with abelian nil-radicals and (skew) multiplicity free actions}\label{sec:abelian.nilrad}

Suppose there is a connected complex Lie group $\itG_\C$ such that $\itG$ is a real form of $\itG_\C$. 
The analytic subgroup $\itL_\C = \exp (\frl\otimes_\R \C)$ is a closed subgroup of $\itG_\C$.  
This follows from the fact that $ \itL_\C $ is a Levi subgroup of the complexified parabolic subgroup $ \itP_\C $.  
We denote the Weyl groups of $ \frg_\C $ and $ \frl_\C $ by $ W_{\frg_\C} $ and $ W_{\frl_\C} $ respectively.  

We say an action of $ \itL_\C $ on an irreducible normal variety $ X $ is \emph{spherical} if a Borel subgroup of $ \itL_\C $ has an open dense orbit in $ X $.  If $ X $ is affine, it is spherical if and only if the regular function ring $ \C[X] $ decomposes without multiplicity as an $ \itL_\C $ module.  In such case, the action of $ \itL_\C $ on $ X $ is called \emph{multiplicity free} (see \cite{Brion.1989} and \cite{Knop.MFA.1998}).  
In particular, for an $\itL_\C$ module $V$, we say $ V $ is a multiplicity free $ \itL_\C $ module if the symmetric algebra $ S(V) $ decomposes multiplicity freely.  
In the same manner, we say an $ \itL_\C $ module $ V $ is \emph{skew multiplicity free} if the exterior algebra $ \twedge V $ decomposes without multiplicity as an $ \itL_\C $ module.  

\begin{thm}\label{thm:SSP.SMF.MF.theorem}
    The following are all equivalent: 
    \begin{enumerate}[label={\makebox[3ex][c]{\upshape(\alph*)}}]
        \item The nilpotent radical $\frn = \Lie(\itN)$ is abelian. \label{pointa}
        \item $\itG/\itL$ is symmetric. \label{pointb}
        \item $ \itK/R = \itG^\theta /\itL^\theta$ is symmetric. \label{pointc}
        \item $\frn_\C = \frn\otimes_\R \C$ is a skew multiplicity free $\itL_\C$ module under the adjoint action. \label{pointd}
        \item The dimension of $\twedge(\frn_\C \oplus \theta \frn_\C)^{\itL_\C}$ is equal to $|W_{\frg_\C}|/|W_{\frl_\C}|$. \label{pointe}
        \item The $ \itL_\C $ module $\twedge(\frn_\C)$ has $|W_{\frg_\C}|/|W_{\frl_\C}|$ irreducible submodules. \label{pointf}
        \item The Casimir $\Cas_{\frl_\C}$ acts by the scalar $||\rho_{\frg_\C}||^2 - ||\rho_{\frl_\C}||^2 $ on $\twedge(\frn_\C \oplus \theta \frn_\C)$.\label{pointg}
        \item The space of cochains associated to the $\frn$ cohomology of the trivial module has zero differential.\label{pointh}
        \item $ \itG_\C/ \itP_\C $ is a compact Hermitian symmetric space.\label{pointi}
    \end{enumerate}
Moreover, if one of the equivalent conditions above is satisfied, the following hold.
\begin{enumerate}[label={\makebox[3ex][c]{\upshape(\alph*)}},start=10]
\item The $ \itL_\C $-module $ \frn_\C $ is multiplicity free (or spherical).\label{pointj}
\item $ \itG_\C/ \itP_\C $ is $ \itL_\C $-spherical.\label{pointk}
\item $ \itL_\C/ \itB_{\itL_\C} \times \itG_\C/ \itP_\C $ is a double flag variety of finite type.\label{pointl}
\end{enumerate}
\end{thm}

We make some remarks before we engage with the proof.

\begin{rem}\label{remark:to.theorem.P.with.abelian.nilradical}
\begin{enumerate}[label={\upshape(\arabic*)}]
\item
If $ \itG_\C $ is simple and $ \frn_\C $ is abelian, $ \itP_\C $ is a maximal parabolic subgroup.  
There is a list of such parabolics in \cite[Remark 2.3]{RRS.1992}.  
In this case, $ \itG_\C/ \itP_\C $ is a compact Hermitian symmetric space.  
\item
Some of equivalent statements in the theorem (especially \ref{pointa}--\ref{pointc}) are well known among experts.  
For this, see \cite{Takeuchi.1962,TK1968}, 
\cite[\S{5.1}]{Howe.SchurLecture.1995}, \cite[\S{1}]{RRS.1992} and references therein.  
See also \cite[Th.{2.2}]{CNP.arxiv2023}.  The proof presented below is somewhat stream-lined.
\item
Skew multiplicity free representations $ V $ of $ \itK $ are completely classified by Howe \cite[\S{4}]{Howe.SchurLecture.1995} for $ \itK $ simple and $ V $ irreducible; a complete classification is given by Pecher \cite{Pecher.2012}.  According to the classification, there exist skew multiplicity free modules which do not come from nilpotent radicals of parabolic subalgebras. 
Compare Theorem 4.7.1 of \cite{Howe.SchurLecture.1995} and the list given in \S~{5.5.1} (and also \S~{5.5.2}) of \cite{Howe.SchurLecture.1995}.  
It is interesting that if $ \frn_\C $ is Heisenberg, then $ \frn_\C/[\frn_\C, \frn_\C] $ is skew multiplicity free for the action of the Levi subgroup (see \cite[\S~{5.5.2}]{Howe.SchurLecture.1995}).  
\item
Multiplicity free representations are classified by Kac \cite{Kac.1980} for $ V $ irreducible, 
Benson-Ratcliff \cite{Benson.Ratcliff.1996}
and Leahy \cite{Leahy.1998} 
including reducible $ V $.  
As indicated in the theorem, if a nilpotent radical $ \frn_\C $ is skew multiplicity free (SMF in short), then it is multiplicity free (MF in short).  However, there are MF nilpotent radicals which are not SMF.  For simple $ \itG_\C $, there are two such examples: 
(a)\ 
For type $ B_n $, the parabolic subalgebra corresponding to the root $ \alpha_n $, in that case 
$ L_\C = \GL_n(\C) $ and 
$ \frn_\C = \C^n \oplus \wedge^2 \C^n $ (the Lie algebra with the bracket $ [(x, z), (y, z')] = (0, x \transpose{y} - y \transpose{x}) $); and 
(b)\ 
For type $ C_n $, the parabolic subalgebra corresponding to the root $ \alpha_1 $, in that case 
$ L_\C = \Sp_{2 n - 2}(\C) \times \GL_1(\C) $ and 
$ \frn_\C = \C^{2 n - 2} $ (a subalgebra of the Heisenberg Lie algebra).  
For the details, see \cite[Tables 1 \& 2 in \S{5.2}]{Avdeev.Petukhov.2020}.

\item
For the double flag varieties of finite type, see \cite{Fresse.Nishiyama.Overview}.  
If $ \itG_\C $ is simple, such double flag varieties of finite type, where one of the pair is a complete flag variety, are classified by \cite{HNOO.2013}.  
So, if $ \itG_\C $ is simple, we can prove \ref{pointl} by looking at the classification.  However, we do not follow this strategy here. 
\end{enumerate}
\end{rem}

\begin{proof}[Proof: \upshape\ref{pointa} $\implies$ \ref{pointb}.]
Suppose that $\frn$ is abelian. Equivalently $[\frn,\frn] =0$, then also $[\theta\frn,\theta\frn]=0$, hence 
\[ 
[\frs,\frs ] = [\frn \oplus \theta \frn,\frn \oplus \theta \frn ] = [\frn,\theta\frn] + [\theta \frn,\frn]
= [\frn,\theta\frn] . 
\]
Let $x,y,z \in \frn$ then $B([x,\theta y],z) = B(\theta y, [z,x]) = 0$, similarly $B([x,\theta y],\theta z) =0$.  
Thence $ [\frs,\frs] = [\frn,\theta \frn]$ is orthogonal to $\frn \oplus \theta\frn = \frs$, 
and thus must be contained in $\frl$. 
Therefore the pair $(\frg,\frl)$ is such that $[\frs,\frs] \subset \frl$, thus is a symmetric pair and $\itG/ \itL$ is also symmetric.
\end{proof}

\begin{proof}[Proof: \upshape\ref{pointb} $\implies$ \ref{pointc}.]
Since $\itG/\itL$ is symmetric there is an involution $\sigma: \frg \to \frg$ such that $\frl = \frg^\sigma$ and $\frs = \frg^{-\sigma}$. To prove that $ \itG^\theta/ \itL^\theta$ is symmetric it is sufficient to prove that $\theta \sigma = \sigma \theta$. Note that $\frs = \frn \oplus \theta \frn$ is $\theta$ stable. Furthermore, $\frl = \frs^\perp$ so it is also $\theta$ stable. 
So $\theta$ preserves the eigenspaces of $\sigma$ and therefore  $\theta$ and $\sigma$ commute.
\end{proof}

\begin{proof}[Proof: \upshape\ref{pointc} $\implies$ \ref{pointa}.]
Since $K/R$ is symmetric, the pair $(\frk,\frr)=(\frl^\theta\oplus\frs^\theta,\frl^\theta)$ is symmetric, in particular 
\eq\label{krsym}
[\frs^\theta,\frs^\theta]\subseteq \frl^\theta.
\eeq

To show $[\frn,\frn]=0$, it is enough to pick $\alpha,\beta\in\Delta(\frn,\fra)$ and $x\in\frg_\alpha$, $y\in\frg_\beta$, and show that $[x,y]=0$. 

Note that $x+\theta x,y+\theta y\in\frs^\theta$, so by \eqref{krsym},
$[x+\theta x,y+\theta y]\in \frl^\theta$. On the other hand,
\begin{equation}\label{krsym2}
[x+\theta x,y+\theta y]=[x,y]+[x, \theta y]+[\theta x,y]+[\theta x,\theta y].
\end{equation}
Recall the semisimple element $ h $ which characterizes $ \lie{n} $ and $ \lie{l} $ (see \eqref{eq:description.of.lie.L}). 
Without loss of generality, we can assume $ \alpha(h) \geq \beta(h) > 0 $.  

Let us consider the case where $ \alpha(h) > \beta(h) > 0 $.  
The four summands in \eqref{krsym2} are root vectors (or zero) and 
$\alpha+\beta$, $\alpha-\beta$, $-\alpha+\beta$, $-\alpha-\beta$ are all different.
It follows that there can be no cancellations in \eqref{krsym2}.
By the assumption $ \alpha(h) > \beta(h) $, we know $ [x, y], [x, \theta y] \in \lie{n} $ and 
consequently the right hand side of \eqref{krsym2} is in $ \lie{s}^{\theta} $, while the left hand side is in $ \lie{l}^{\theta} $ as we already mentioned.  
So it must be zero.  Thus all the component must be zero.

Next let us consider the case where $ \alpha(h) = \beta(h) > 0 $.  
In this case, $ [x,  y] + \theta([x, y]) $ is in $ \lie{s}^{\theta} $, while $ [x, \theta y] + \theta([x, \theta y]) $ is in $ \lie{l}^{\theta} $.  
Thus $ [x,  y] + \theta([x, y]) = 0 $, which implies $ [x, y] = 0 $ since $ \alpha + \beta $ and $ - \alpha - \beta $ are different.
%
\end{proof}

We postpone the proof of the rest of the claims to the next section after preparing the basics on Dirac cohomology.

\section{Dirac cohomology and multiplicity one submodules}\label{sec: mult1}

To prove the rest of the statements of Theorem~\ref{thm:SSP.SMF.MF.theorem},
we first introduce Dirac cohomology. 
Recall $\frg_\C = \frl_\C \oplus \frs_\C$ and 
$ \frs_\C = \frn_\C \oplus \theta \frn_\C $.  
We choose the Lagrangian subspace of $\frs_\C$, with respect to $ B(x, y) $, to be $ \frs_\C^+ = \frn_\C $.  
Let  $ \Cl({\frs_\C}) $ be the Clifford algebra of $\frs_\bbC$; it is an associative algebra with unit, generated by $ \frs_\C $ and subject to the relations $ x y + y x = 2 B(x, y) \; (x, y \in \frs_\C) $. Let $ S = \twedge \frs_\C^+ $ be the spin module for $\Cl({\frs_\C})$.  
The Goette-Kostant cubic Dirac operator corresponding to the pair $({\frg_\C},{\frl_\C})$ is defined as
\[
D=\sum_i b_i\otimes d_i +\half \otimes \sum_{i,j} b_i b_j[d_i,d_j]\quad\in U({\frg_\C})\otimes \Cl({\frs_\C}),
\]
where $ \{ b_i \}_i $ is a basis of $ \frs_\C $ and $ \{ d_i \}_i $ is the dual basis with respect to $ B(x, y) $ (\cite{Kostant.cubic.1999,Goette.MZ.1999}).  
The last cubic term belongs to $\Cl({\frs_\C})$ and acts on the spin module.
We denote 
\begin{equation}
c = \sum_{i,j} b_i b_j[d_i,d_j] 
\end{equation}
without the factor $ 1/2 $ and also call it the cubic term.  
$D$ and $ c $ are easily seen to be independent of the choice of the basis $ \{ b_i \}_i $ and ${\frl_\C}$-invariant for the diagonal action:
\begin{equation}\label{eq:DeltaX.spin.action}
\Delta(X)=X\otimes 1+1\otimes\alpha(X) \qquad (X\in{\frl_\C}), 
\end{equation}
where $ \alpha(X) = \dfrac{1}{4} \sum_i [X, b_i] d_i $ is the natural embedding map of $ \frl_\C $ into $ \Cl({\frs_\C}) $  
(the spin embedding of $ \frl_\C $).  
Furthermore, due to \cite[Theorem 2.16]{Kostant.cubic.1999}, the square of $D$ is
\begin{equation}\label{eq:Dsquare}
D^2=(\Cas_{\frg_\C}+\|\rho_{\frg_\C}\|^2)-(\Cas_{(\frl_\C)_\Delta}+\|\rho_{\frl_\C}\|^2) ,
\end{equation}
where $\Cas_{\frg_\C}$ is the Casimir element for ${\frg_\C}$, and $\Cas_{(\frl_\C)_\Delta}$ is the Casimir element for the diagonal copy $(\frl_\C)_\Delta$ of ${\frl_\C}$, obtained as the image of the map \eqref{eq:DeltaX.spin.action}.  

Let $F$ be an irreducible finite-dimensional ${\frg_\C}$-module. Then $D$ acts on $F\otimes S$, and moreover it exchanges $F\otimes S^+$ and $F\otimes S^-$, where $S^\pm$ denote the even respectively odd parts of the spin module $S=\twedge{\frs_\C^+}$. The Dirac cohomology of $F$ is the ${\frl_\C}$-module
\[
H_D(F)=\ker D/\im D\cap\ker D.
\]
As proved in \cite[Remark 4.8]{HPRarXiv}, there is a positive definite inner product on $F\otimes S$ such that $D$ is self adjoint with respect to this inner product. This implies
\[
H_D(F)=\ker D=\ker D^2.
\]
Moreover, we have an analogue of Vogan's conjecture, which says that the ${\frl_\C}$-infinitesimal character of any ${\frl_\C}$-submodule of $H_D(F)$ is conjugate to the ${\frg_\C}$-infinitesimal character of $F$ under the Weyl group $W_{\frg_\C}$ (see \cite[Theorem 4.1.5]{HPbook}). This leads to the following result of Kostant (see \cite[Theorem 4.2.2]{HPbook}). 

\begin{thm}\label{HD formula}
Let $ \frt_\C \subset \frl_\C $ be a Cartan subalgebra of $\frg_\C$ (and $\frl_\C$). Choose compatible positive root systems for ${\frg_\C}$ and ${\frl_\C}$ with respect to $ \frt_\C$, such that $\Delta^+({\frg_\C},\frt_\C)=\Delta^+({\frl_\C},\frt_\C)\cup\Delta({\frs_\C}^+)$.  Let
\[
W^1=\{w\in W_{\frg_\C}\bbar w\rho_{\frg_\C} \text{ is ${\frl_\C}$-dominant}\},
\]
which is the set of minimal length representatives of $W_{\frl_\C}\backslash W_{\frg_\C}$.  
For a ${\frg_\C}$-dominant weight $\lambda$, let $F_\lambda$ be the irreducible finite-dimensional ${\frg_\C}$-module with highest weight $\lambda$. 
For an ${\frl_\C}$-dominant weight $\mu$, let $E_\mu$ denote the irreducible ${\frl_\C}$-module with highest weight $\mu$.

Then the Dirac cohomology of $F_\lambda$ is equal to
\begin{equation}\label{eq:decomposition.formula.HD}
H_D(F_\lambda)= \Ker D = \Ker D^2 = \bigoplus_{w\in W^1} E_{w(\lambda+\rho_{\frg_\C})-\rho_{\frl_\C}}.
\end{equation}
Moreover, each of the weights $w(\lambda+\rho_{\frg_\C})-\rho_{\frl_\C}$ of $F_\la\otimes S$ has multiplicity one.
\end{thm}

From the proof described in \cite{HPbook}, it is visible that each $E_{w(\lambda+\rho_{\frg_\C})-\rho_{\frl_\C}}$ contributing to $H_D(F_\lambda)$ is of multiplicity one not only in $H_D(F_\lambda)$, but also in $F_\la\otimes S$. Since both the weight $w\lambda$ of $F_\la$ and the weight $w\rho_{\frg_\C}-\rho_{\frl_\C}$ of $S$ are the unique highest weights with respect to $w\Delta^+({\frg_\C},\frt)$, their sum has weight multiplicity equal to one.

On the other hand, if an ${\frl_\C}$-submodule $E_\mu$ of $F_\la\otimes S$ has multiplicity one, then it must be in $\ker D=H_D(F_\la)$. Namely, suppose $DE_\mu\neq 0$. Then $DE_\mu$ is an ${\frl_\C}$-submodule of $F_\la\otimes S$ isomorphic to $E_\mu$, but it can not be equal to $E_\mu$ since $D$ changes parity, so $E_\mu\subset F_\la\otimes S^\pm$ implies $DE_\mu\subset F_\la\otimes S^\mp$. This contradicts the assumption that $E_\mu$ is of multiplicity one. We have proved:

\begin{cor}\label{cor mult 1 F}
$H_D(F_\la)=\ker D\subset F_\la\otimes S$ consists precisely of the irreducible 
${\frl_\C}$-submodules of $F_\la\otimes S$ that have multiplicity one.
\end{cor}

In particular, for the trivial module $F_\la=F_0=\C$ we get

\begin{cor}\label{cor mult 1 triv}
The multiplicity one ${\frl_\C}$-submodules of the spin module $S=\twedge{\frs_\C}^+$ are precisely the modules $E_{w\rho_{\frg_\C}-\rho_{\frl_\C}}$, $w\in W^1$, which comprise the Dirac cohomology $H_D(\C)$.
\end{cor}

\begin{cor}\label{cor:SpinMF when c is zero}
The spin module $S$ is multiplicity free exactly when $H_D(\C) = S$, or equivalently, when the cubic term $ c = \sum_{i,j} b_i b_j[d_i,d_j] =0$.
\end{cor}

\begin{proof}
Since the spin module $ S $ is simple and $\Cl(\frs_\bbC)=\End S$, $c\in \Cl(\frs_\bbC)$ is zero iff it acts by $0$ as an endomorphism of $S$.
The mixed terms (degree two terms) $b_i\otimes d_i$ 
act on $ S \otimes F_0 $ trivially because $ F_0 $ is the trivial module, 
therefore $ D $ acts on $ S $ trivially if and only if $ c = 0 $.
\end{proof}

\begin{lem}\label{lem:SpinMF iff symmetri pair}
The spin module $S$ is multiplicity free if and only if $(\frg,\frl)$ is a symmetric pair.
\end{lem}

\begin{proof}
Using Corollary \ref{cor:SpinMF when c is zero}, we argue that $c=0$ if and only if $(\frg,\frl)$ is a symmetric pair.   
We can rewrite the cubic term $c$ as $\sum_{i,j,k} B(b_i,[b_j,b_k]) d_i d_j d_k$. Suppose that $c$ is zero, then for every $i,j,k$ we must have $B(b_i,[b_j,b_k]) =0$. Thus $[\frs_\C,\frs_\C] \perp \frs_\C$, and we conclude $[\frs_\C,\frs_\C] \subset \frl_\C $. Equivalently $(\frg_\C,\frl_\C)$ is a symmetric pair. 
Since $\frg$ is a real form of $\frg_\C$ and the involution defining $(\frg_\C,\frl_\C)$ as a symmetric pair is $\C$ linear, it follows that this involution restricts to an involution of $\frg$ such that $\frl$ is the fixed space of this involution. Hence $(\frg,\frl)$ is a symmetric pair.

Conversely, suppose that $(\frg,\frl)$ is a symmetric pair, then $(\frg_\C,\frl_\C)$ is a symmetric pair and 
$[\frs_\C,\frs_\C] \perp \frs_\C$.  
Thus  $c=\sum_{i,j,k} B(b_i,[b_j,b_k]) d_i d_j d_k$ is zero. 
\end{proof}

\skipover{
To end this section, we explicitly describe the highest weight vectors for the above modules $E_{w(\la+\rho_{\frg_\C})-\rho_{\frl_\C}}$. First, it is clear that the restriction of $F_\lambda$ to ${\frl_\C}$ contains $E_{w\lambda}$ with multiplicity one, for each $w\in W^1$; namely, $w\lambda$, $w\in W^1$ are extremal weights of $F_\lambda$, and they are also ${\frl_\C}$-dominant. Moreover, the vector $f_{w\la}$ of weight $w\lambda$ (which is unique up to scalar) is a highest weight vector for ${\frl_\C}$, since it is killed by all root vectors for roots in $w\Delta^+({\frg_\C},\frt)\supset\Delta^+({\frl_\C},\frt)$.

On the other hand, the spin module $S=\C\otimes S$ contains unique (up to scalar) vectors of weight $w\rho_{\frg_\C}-\rho_{\frl_\C}$; these are the highest weight vectors for the Dirac cohomology $H_D(\C)$. To write these vectors explicitly, we note that the corresponding adjoint weight is $w\rho_{\frg_\C}-\rho_{\frl_\C}+\rho({\frs_\C}^+)$. 
Note that
\[
w\Delta^+({\frg_\C},\frt) = \Delta^+({\frl_\C},\frt) \cup [w\Delta^+({\frg_\C},\frt)\cap \Delta({\frs_\C}^+)]\cup 
[(w\Delta^+({\frg_\C},\frt)\cap(-\Delta({\frs_\C}^+)))].
\]
If we denote $w\Delta^+({\frg_\C},\frt)\cap \Delta({\frs_\C}^+)$ by $\Delta_w$ and $w\Delta^+({\frg_\C},\frt)\cap (-\Delta({\frs_\C}^+))$ by $\Delta_w'$, we see that
\[
w\rho_{\frg_\C}-\rho_{\frl_\C}+\rho({\frs_\C}^+)=  \sum_{\alpha\in\Delta_w}\alpha.
\]
Namely, roots in $\Delta^+({\frl_\C},\frt)$ appear in $w\rho_{\frg_\C}-\rho_{\frl_\C}+\rho({\frs_\C}^+)$ twice, with coefficients $\half$ in $ w\rho_{\frg_\C} $ and $-\half$ in $ -\rho_{\frl_\C} $, as do the roots in $\Delta_w'$ 
(they appear in $ \half $ for $ w\rho_{\frg_\C} $ and $ -\half $ for $ \rho({\frs_\C}^+) $), hence they cancel.  
The roots in $\Delta_w$ appear twice with both coefficients equal to $\half$ in $w\rho_{\frg_\C}$ and $\rho({\frs_\C}^+)$, and they add up to $ \alpha $.

It is now clear that up to scalar, the vector in $S=\twedge{\frs_\C}^+$ of adjoint weight $w\rho_{\frg_\C}-\rho_{\frl_\C}+\rho({\frs_\C}^+)$ (and hence of spin weight $w\rho_{\frg_\C}-\rho_{\frl_\C}$) is
\[
\twedge_{\alpha\in\Delta_w} e_\alpha,
\]
where $e_\alpha\in{\frs_\C}^+$ denotes the root vector for the root $\alpha$. To summarize, we have proved

\begin{cor}\label{hwt vectors}
Up to scalar, the highest weight vector for the ${\frl_\C}$-submodule $E_{w(\la+\rho_{\frg_\C})-\rho_{\frl_\C}}$ is
\[
f_{w\la}\otimes \twedge_{\alpha\in\Delta_w} e_\alpha
\]
(with notation as above).
\end{cor}

}

\section{The proof of the equivalence of the rest of the statements in Theorem~\ref{thm:SSP.SMF.MF.theorem}}

\begin{proof}[Proof: \upshape\ref{pointb} $\iff$ \ref{pointd}.]
The spin module $S=\twedge\frn_\C$ has two $\frl_\C$-actions, the spin action and the adjoint action, which differ only by tensoring with a  one-dimensional character of weight  $ \rho_{\frn_\C} $ (see \cite[Prop.{3.6}]{Kostant.LMP.2000}). 
It follows that  $\frn_\bbC$ is a skew multiplicity free  $\frl_\C$ module if and only if the spin module $S$ is multiplicity free. By Lemma \ref{lem:SpinMF iff symmetri pair}, the latter is equivalent to $(\frg_\C,\frl_\C)$ being a symmetric pair, which is in turn equivalent to $G/L$ being symmetric.
\end{proof}

\begin{proof}[Proof: \upshape\ref{pointb} $\implies$ \ref{pointf}.]
If $\itG/\itL$ is symmetric then $\twedge(\frn_\C)$ is the spin module for $\frl_\C$ and by Lemma \ref{lem:SpinMF iff symmetri pair} and Corollary \ref{cor mult 1 triv} 
the $ \frl_\C $ module $\twedge(\frn_\C)$ has $|W^1| = |W_{\frg_\C}|/|W_{\frl_\C}|$ irreducible components. 
\end{proof}

\begin{proof}[Proof: \upshape\ref{pointf} $\implies$ \ref{pointe}.]
Suppose $\twedge(\frn_\C)$ has $|W^1| = |W_{\frg_\C}|/|W_{\frl_\C}|$ irreducible components. 
Then by Corollary \ref{cor mult 1 triv} 
these must be exactly the modules $E_{w\rho_{\frg_\C}-\rho_{\frl_\C}}$, $w\in W^1$; in particular, they are all  different from each other.

Note that $\twedge(\theta \frn_\C) \cong \twedge(\frn_\C)^*$ and $\twedge(\frn_\C \oplus \theta \frn_\C) \cong \twedge(\frn_\C) \otimes \twedge(\theta \frn_\C)$. Thus 
\[ \twedge(\frn_\C \oplus \theta \frn_\C)^{\frl_\C} \cong \End_{\frl_\C}(\twedge(\frn_\C)).\]
It now follows from Schur's lemma that  $\End_{\frl_\C}(\twedge(\frn_\C))$ is spanned by the projections to various $E_{w\rho_{\frg_\C}-\rho_{\frl_\C}}$; in particular, it is of dimension $|W^1|$.
\end{proof}

\begin{proof}[Proof: \upshape\ref{pointe} $\implies$ \ref{pointd}]
We prove the contraposition.  
So suppose that $\frn_\C$ is not skew multiplicity free. Then the Dirac cohomology $H_D(\C)$ is strictly contained in $\twedge(\frn_\C)$. 
On the other hand, by Corollary \ref{cor mult 1 triv}, $H_D(\C)$ has $|W^1|$ different irreducible submodules. Hence $\twedge(\frn_\C)$ has strictly more than $|W^1|$ different submodules and thus the dimension of the endomorphism space $\End_{\frl_\C}(\twedge(\frn_\C))$ is strictly greater than $|W^1|$.
\end{proof}

So far, we established the equivalence of the statements \ref{pointa}--\ref{pointf}.  

\begin{proof}[Proof: \upshape\ref{pointb} $\iff$ \ref{pointg}]
By Corollary \ref{cor:SpinMF when c is zero} and Lemma \ref{lem:SpinMF iff symmetri pair}, the space $\itG/\itL$ is symmetric if and only if $H_D(\C) = \twedge(\frn_\C)$. This implies that the spin action of $\frl_\C$ is such that $\Cas_{\frl_\C}$ acts on $\twedge(\frn_\C)$ by $||\rho_{\frg_\C}||^2 - ||\rho_{\frl_\C}||^2$ (see Corollary~\ref{cor mult 1 triv} and \eqref{eq:Dsquare}). The adjoint action and the spin action differ by a weight shift of $\rho_{\frn_\C}$ on $ \twedge(\frn_\C) $. 
However, the actions also differ on $ \twedge(\theta\frn_\C) $ by a weight shift of $- \rho_{\frn_\C}$, and these shifts cancel on $ \twedge(\frn_\C) \otimes \twedge(\theta\frn_\C) \simeq \twedge(\frn_\C \oplus \theta\frn_\C) $.  
Thus the spin action and the adjoint action agree and we conclude that $\itG/\itL$ is symmetric if and only if $\ad(\Cas_{\frl_\C})$ acts by $||\rho_{\frg_\C}||^2 - ||\rho_{\frl_\C}||^2 $.
\end{proof}

For the definition of the cubic term $ c $, we can use any basis $ \{ b_i \}_i $ of $ \frs_\C = \frn_\C \oplus \theta\frn_\C $.  
In particular, we can take a basis consisting of root vectors $ \{ e_{\alpha} \mid \alpha \in \Delta(\frg, \fra) \} $, repeated according to multiplicity, and chosen so that $\{ e_\alpha\}$ and $\{ e_{-\alpha}\}$ form dual bases with respect to $B$. Then the cubic term can be written as
\begin{equation}
c = \sum_{\alpha, \beta, \gamma \in \Delta} B(e_{-\alpha}, [e_{-\beta}, e_{-\gamma}]) \, e_{\alpha} e_{\beta} e_{\gamma} .
\end{equation}
Now note that $ B(e_{-\alpha}, [e_{-\beta}, e_{-\gamma}]) = 0 $ unless $ \alpha + \beta + \gamma = 0 $.  
This means the term $ e_{\alpha} e_{\beta} e_{\gamma} $ only appears when $ \alpha + \beta + \gamma = 0 $.  
We divide the terms into two types; 
(a)\ 
two of the terms are positive root vectors and one is negative, and 
(b)\ 
two of the terms are negative root vectors and one is positive.  
The sum of the terms of type (a) is denoted by $ c^+ $ and the sum of the terms of type (b) is denoted by $ c^- $, 
so that we have $ c = c^+ + c^- $.  

If we identify the spin module $ S $ as $ \twedge(\frn_\bbC) $, 
the multiplication by $c^-$ amounts to the differential $ d_{\wedge} $ on the space of cochains for the Lie algebra cohomology of $\frn_\C$ and multiplication by $c^+$ amounts to its dual $ \partial_{\wedge} $.  
The $\frn$ cohomology is given by the kernel of the Laplacian $ L = [c^+,c^-]$ on $\twedge(\frn_\C)$. Note that $(c^+)^2 = 0$ since it is a differential and $(c^-)^2=0$ since it is dual to $c^+$. Thus $c^2 = [c^+,c^-] = L $. For more details see \cite{Kostant.I.1961}, \cite{HPbook} or \cite{HPRarXiv}.

\begin{proof}[Proof: \upshape\ref{pointb} $\iff$ \ref{pointh}]
Note that $\itG/\itL$ is symmetric if and only if the cubic element $c$ for the pair $(\frg_\C,\frl_\C)$ is zero by Corollary~\ref{cor:SpinMF when c is zero} (and the claim \ref{pointd}). Which is equivalent to $c^2 = 0$ (since $\C$ is finite dimensional). Thus $[c^+,c^-] = 0$ on the whole $ \twedge(\frn_\C) $ is equivalent to $\itG/\itL$ symmetric. 
Since $ L = [c^+,c^-] $ is the Laplacian, by the Hodge decomposition, we find that $L=[c^+,c^-]=0$ is equivalent to $d_{\wedge} = 0$.
\end{proof}

\begin{proof}[Proof: \upshape\ref{pointa} $\iff$ \ref{pointi}]
$ \itG_\C/ \itP_\C $ is Hermitian symmetric if and only if every irreducible factor is so.  
Hence we can assume $ \itG_\C $ is simple.  
In this case, it is well known that the parabolic $ \frp_\C $ has abelian nilpotent radical if and only if 
it is corresponding to the simple roots $ \Pi \setminus \{ \alpha_0 \} $, where $ \alpha_0 $ has height $ 1 $ in the highest root.  
Such parabolic subalgebra is called cominuscule, and that characterizes Hermitian symmetric spaces.  
See Lemma 2.2 and \S{1}, Introduction of \cite{RRS.1992}.
\end{proof}

\begin{proof}[Proof: \upshape\ref{pointa} $\implies$ \ref{pointj}, \ref{pointk}, \ref{pointl}]
Let us assume $ \frn_\C $ is abelian.  
Take a Borel subalgebra $ \frb_{\frg_\C} \subset \frg_\C $ which contains $ \frn_\C $ and a Borel subalgebra $ \frb_{\frl_\C} $ of $ \frl_\C $, so that 
$ \frb_{\frg_\C} = \frb_{\frl_\C} \oplus \frn_\C $, and consider the corresponding Borel subgroups $ \itB_{\itG_\C} \subset \itG_\C $ and $ \itB_{\itL_\C} \subset \itL_\C $.  
Since $ \itG_\C/ \itP_\C $ has finitely many $ \itB_{\itG_\C} $-orbits, 
$ \frn_\C $ has finitely many $ \Ad(\itB_{\itG_\C}) $ orbits.  
Since $ \itB_{\itG_\C} = \itB_{\itL_\C} \ltimes \itN_\C $ is a semidirect product and 
$ \itN_\C $ acts on $ \frn_\C $ trivially, because $ \frn_\C $ is abelian, 
the number of $ \itB_{\itG_\C} $-orbits in $ \frn_\C $ and that of $ \itB_{\itL_\C} $-orbits are the same.  
Hence $ \itB_{\itL_\C} $ has finitely many orbits in $ \frn_\C $.  
In particular, there is an open dense $ \itB_{\itL_\C} $ orbit, which implies $ \frn_\C $ is a spherical $ \itL_\C $ variety.  
Thus, by \cite[Theorem 3.1]{Knop.MFA.1998}, the symmetric algebra $ S(\frn_\C) $ decomposes without multiplicity.  
This proves \ref{pointj}.

Since $ \itB_{\itL_\C} $ has an open dense orbit in $ \frn_\C $ and there is an $ \itL_\C $-equivariant open embedding of 
$ \frn_\C $ into $ \itG_\C/ \itP_\C $ via the exponential map, 
$ \itB_{\itL_\C} $ has an open dense orbit in $ \itG_\C/ \itP_\C $.   
This means $ \itG_\C/ \itP_\C $ is $ \itL_\C $-spherical, i.e., \ref{pointk} holds.

Finally, the $ \itL_\C $ orbits on the double flag variety 
$ \itL_\C/ \itB_{\itL_\C} \times \itG_\C/ \itP_\C $ are parametrized by the double coset space 
$ \itB_{\itL_\C} \backslash \itG_\C/ \itP_\C $.  
From \ref{pointk}, it has finitely many double cosets and the double flag variety is of finite type, i.e., \ref{pointl} holds. 
\end{proof}

\section{Application to  de Rham cohomology}\label{sec:deRham.cohomology}

\subsection{Theorems of Cartan-de Rham and Cartan-Borel}

We begin with two classical theorems.  
In the following two theorems, $ K $ and $ R $ are general compact Lie groups satisfying the assumptions of the theorems.

\begin{thm}[Cartan-de Rham]\label{thm:Cartan-deRham}
Let $\itK$ be a connected compact Lie group and $ R $ a closed subgroup.  
We assume $\itK/R$ is a symmetric space.  
Then de Rham cohomology ring of $ \itK/R $ is given by
\begin{equation*}
H_{dR}(\itK/R) \simeq (\twedge(\frk_\C/\frr_\C)^\ast)^R .
\end{equation*} 
%
\end{thm}

A short and comprehensive proof of this theorem can be found in \cite[\S{4.1}]{CNP.arxiv2023}, for example.  

Let $ T \subset R $ be a maximal torus, and let us denote the Lie algebra of $T$ by $ \frt $.  
The Weyl groups $ W_{\frk_\C} = W(\frk_\C, \frt_\C) $ and $ W_{\frr_\C} = W(\frr_\C, \frt_\C) $ are considered with respect to this torus.   
Note that $ \rank \frk_\C $ might be strictly larger than $ \rank \frr_\C $ so that $ \frt_\C $ might not be a Cartan subalgebra of $ \frk_\C $.  In this case, we refer to the Weyl group $ W(\frk_\C, \frt_\C) $ as the reflection group generated by the restricted roots 
$ \Delta(\frk_\C, \frt_\C) $.

Let $\mathcal{P}_{\frk_\C}$ be a space of primitives generating $S(\frk_\C^\ast)^{\frk_\C}$, there is a well known transgression map (\cite{CCC.III.1976} and \cite{CGKP.2025}), $t_{\frk_\C} : S(\frk_\C^\ast)^{\frk_\C} \to \twedge(\frk_\C^\ast)^{\frk_\C}$, and the primitives for $\twedge(\frk_\C^\ast)^{\frk_\C}$ which are linearly independent generator of this algebra, are given by $t_{\frk_\C}(\mathcal{P}_{\frk_\C})$ and denoted $\bbP_{\frk_\C}$. There is an analogous transgression map for the relative algebra $\twedge((\frk_\C/\frr_\C)^\ast)$, 
\[ 
t_{\frk_\C/\frr_\C} : S(\frk_\C^\ast)^{\frk_\C} \to \twedge((\frk_\C /\frr_\C)^\ast)^{\frk_\C}.
\]
The relative space of primitives $\bbP_{\frk_\C/\frr_\C}$ are defined as $t_{\frk_\C/\frr_\C}(\mathcal{P}_{\frk_\C})$ 
(see \cite[pg 427]{CCC.III.1976} and \cite[\S 4.3]{CGKP.2025}).  
The pair $(\frk_\C,\frr_\C)$ is called a Cartan pair if 
$ \bbP_{\frk_\C/\frr_\C} 
= \dim G - \dim R $ (\cite[pg 431]{CCC.III.1976}), 
and symmetric pairs are Cartan pairs (\cite[pg 448]{CCC.III.1976}).  
We refer the reader to \cite{Onishchik.1994} for a modern exposition of this material.

Now we quote the second classical theorem.   
See \cite[pg. 432]{CCC.III.1976}; for symmetric pairs see also \cite{CGKP.2025} and this case suffices for our present purpose.

\begin{thm}[Cartan-Borel]\label{thm:Cartan-Borel}
Let $K$ be a compact connected Lie group, and let $R$ be a closed connected subgroup of $K$, such that $(\frk,\frr)$ is a Cartan pair. Then 
\begin{equation*}
H_{dR}(\itK/R) \simeq \twedge(\bbP_{\frk_\C/\frr_\C}) \otimes \Bigl( S(\frt_\C^*)^{W_{\frr_\C}}/ \langle S_+(\frt_\C^*)^{W_{\frk_\C}} \rangle \Bigr)
\end{equation*}
where $ S_+(\frt_\C^*) $ is the augmentation ideal of positive degree polynomials and 
$ \bbP_{\frk_\C/\frr_\C} $ is the space of primitives, on which $ \frr_\C $ acts trivially. 
Moreover $ \dim\bbP_{\frk_\C/\frr_\C} = \rank \frk_\C - \rank \frr_\C $.  
\end{thm}

\subsection{Consequence of the skew multiplicity free action}

We go back to the setting in \S\ref{sec:abelian.nilrad} and assume that the equivalent conditions in Theorem~\ref{thm:SSP.SMF.MF.theorem} hold. 
Recall that $ G/P \simeq K/R $ with $ R = \itK^{\sigma} = \itL^{\theta} $, 
and $ \frr = \frk^{\sigma} = \frl^{\theta} $ is a symmetric subalgebra of $ \frk $ and $ \frl $ at the same time.  
Thus, $ R $ is a maximal compact subgroup of $ L $ (disconnected in general) and $ (\frl_\C, \frr_\C) $ is a symmetric pair. 
Let us denote $ \frq = \lie{k}^{-\sigma} = \{ x \in \lie{k} \mid \sigma(x) = - x \} $, which is isomorphic to 
the tangent space $ T_{eR}(\itK/R) $ as an $ R $-module.  

From Theorems \ref{thm:Cartan-deRham} and \ref{thm:Cartan-Borel}, we get 

\begin{cor}
Suppose $G$ (and hence $K$) is connected. Then 
\begin{equation*}
H_{dR}(\itG/\itP) = H_{dR}(\itK/R) \simeq (\twedge(\frq_\C))^R .
\end{equation*}
Moreover, if $ R $ is also connected, 
\begin{equation*}
\dim H_{dR}(\itG/\itP) = \dim H_{dR}(\itK/R) = 2^{\rank \frk_\C - \rank \frr_\C} |W_{\frk_\C}/W_{\frr_\C}| .
\end{equation*}
\end{cor}
(We will extend this statement to pairs with disconnected $R$ later in this section, see Theorem \ref{thm:dim.Poincare.series.of.cohomology}.)
\begin{lem}\label{lemma:qC.isomorphic.to.nC}
$ \frq_\C \simeq \frn_\C $ as $ \frr_\C $-modules.
\end{lem}

\begin{proof}
Notice that $ \frg^{-\sigma} = \frs = \frn \oplus \theta \frn $.  
Then $ \frn \ni x \mapsto x + \theta x \in (\frg^{-\sigma})^{\theta} = (\frg^{\theta})^{-\sigma} = \frq $ is an $ \frr $ module isomorphism. 
By complexification, we get the isomorphism of the lemma.
\end{proof}

For the Lie algebra cohomology of $ \frn_\C $, Kostant \cite[Theorem 5.14]{Kostant.I.1961} described it completely.  
Let $ W^1 = W^1_{\frg_\C,\frl_\C} = W_{\frl_\C} \backslash W_{\frg_\C} $ be the set of minimal length representatives. 

\begin{thm}[Kostant]\label{thm:Kostant.n.cohomology}
As an $ \Ad(\itL_\C) $-module, 
\begin{equation}\label{eq:Kostant.formula.n.homology}
H(\frn_\C, \C) \simeq \bigoplus_{w \in W^1} E_{w \rho_{\frg_\C} - \rho_{\frg_\C}} .
\end{equation}
Moreover, the representation $ E_{w \rho_{\frg_\C} - \rho_{\frg_\C}} $ lives in the space of degree $ \ell(w) $, the length of $ w \in W_{\frg_\C} $.
\end{thm}

The formula \eqref{eq:Kostant.formula.n.homology} is the shift by the character $\rho_{\frn_\bbC}$ from the formula \eqref{eq:decomposition.formula.HD} describing  $H_D(\bbC)$.  
Since we are in the case where $ \itG_\C/\itL_\C $ is symmetric, the differential $ d_{\wedge} $ vanishes and we get

\begin{cor}\label{cor:Lmodule.structure.of.abelian.wedge.n}
If $ \frn_\C $ is abelian, 
\begin{equation*}
\twedge(\frn_\C) \simeq \bigoplus_{w \in W^1} E_{w \rho_{\frg_\C} - \rho_{\frl_\C} + \rho_{\frn_\C}}.
\end{equation*}
\end{cor}

\begin{proof}
Since the differential vanishes $ H(\frn_\C, \C) \simeq \twedge(\frn_\C^*) $.  
Let us denote the longest elements in $ W_{\frg_\C} $ and $ W_{\frl_\C} $ by $ \wzerog  $ and $ \wzerol  $ respectively.  
Note that we have 
$\wzerog \rho_{\frg_\bbC}=-\rho_{\frg_\bbC}$, $\wzerol \rho_{\frl_\bbC}=-\rho_{\frl_\bbC}$, and $\wzerol \rho_{\frn_\bbC}=\rho_{\frn_\bbC}$, the last one following from the fact that $\rho_{\frn_\bbC}$ is orthogonal to the roots of $\frl_\bbC$.  
Using these formula, we see the highest weight of the dual representation $ E_{w \rho_{\frg_\C} - \rho_{\frg_\C}}^* $ is 
\begin{align*}
- \wzerol  (w \rho_{\frg_\C} - \rho_{\frg_\C})
&= \wzerol  w \wzerog  \rho_{\frg_\C} + \wzerol  (\rho_{\frl_\C} + \rho_{\frn_\C}) 
\\
&= \wzerol 	w \wzerog  \rho_{\frg_\C} - \rho_{\frl_\C} + \rho_{\frn_\C} .
\end{align*}
Note that $ w \in W^1 $ if and only if $ w $ satisfies the equivalent conditions 
\begin{equation*}
\Delta_{\frl_\C}^+ \subset w \Delta_{\frg_\C}^+ 
\iff 
\Delta_{\frl_\C}^+ \subset \wzerol  w \wzerog  \Delta_{\frg_\C}^+.
\end{equation*}
This means, as $w$ runs over $W^1$ so does $  \wzerol  w \wzerog  $. 
\end{proof}

\begin{rem}
Note that the lowest weight of $ E_{w \rho_{\frg_\C} - \rho_{\frl_\C} + \rho_{\frn_\C}} $ is $ \rho_{\frg_\C} - \sigma \rho_{\frg_\C} $, where 
$ \sigma = \wzerol  \, w\, \wzerog  $, which was used by Kostant (see \cite[Cor.{8.1}]{Kostant.I.1961}). 
\end{rem}

From Lemma~\ref{lemma:qC.isomorphic.to.nC} and Corollary~\ref{cor:Lmodule.structure.of.abelian.wedge.n}, we get

\begin{cor}\label{cor:R-invariants.in.wedge.n}
\begin{equation*}
(\twedge(\frq_\C))^{R} \simeq \bigoplus_{w \in W^1} E_{w \rho_{\frg_\C} - \rho_{\frl_\C} + \rho_{\frn_\C}}^{R} .
\end{equation*}
\end{cor}

Since $ \itL/R $ is symmetric, 
we know $ \dim E_{\lambda}^{R} \leq 1 $.  

Let $\frt_\C$ be a Cartan subalgebra of $\frr_\C$.  
We define the group theoretic Weyl group 
\begin{equation*}\label{gp W gp}
W_R=N_R(\frt_\C)/Z_R(\frt_\C) .
\end{equation*}
Note that if $R$ is connected then $W_R = W_{\frr_\C}$, \cite[Theorem 4.54]{beyond}.

Let 
\begin{equation}\label{eq:def.W1(R)}
W^1(R) = \{ w \in W^1_{\frg_\C,\frl_\C} \mid \dim E_{w \rho_{\frg_\C} - \rho_{\frl_\C} + \rho_{\frn_\C}}^{R} = 1 \} .
\end{equation}
%
We can prove the Cartan-Borel theorem for disconnected $R$, of which we postpone the proof to Section \ref{sec:CartanBorelforDisconnected}. Comparing the Cartan-Borel Theorem~\ref{thm:Cartan-Borel} and Theorem \ref{thm:CartanBorelforDisconnected}, Corollary~\ref{cor:R-invariants.in.wedge.n},
and also Theorem 5.1, we now find

\begin{thm}\label{thm:dim.Poincare.series.of.cohomology}
Let $G, K$ and $R$ be as above, then
\begin{equation*}
\dim H_{dR}(\itG/\itP) = \dim H_{dR}(\itK/R) = 2^{\rank \frk_\C - \rank \frr_\C} |W_{\frk_\C}/W_R| = |W^1(R)| .
\end{equation*}
Moreover, we get the Poincar\'{e} series of the cohomology ring:  
\begin{equation}
P(H_{dR}(\itG/\itP), q) = P(H_{dR}(\itK/R), q) = \sum_{w \in W^1(R)} q^{\ell(w)} .
\end{equation}
Here $ \ell(w) $ denotes the length of the element $ w \in W_{\frg_\C} $.
\end{thm}

\begin{proof}
The statement for the Poincar\'{e} series comes from Kostant's theorem (Theorem~\ref{thm:Kostant.n.cohomology}) and the definition \eqref{eq:def.W1(R)} of 
$ W^1(R) $.   
\end{proof}

\section{de Rham cohomology for disconnected $R$}\label{sec:CartanBorelforDisconnected}

Let $b_i$ and $d_i$ be dual bases of $\frq_\C$, define $\lambda: S(\frr_\C) \to \twedge(\frq_\C) \cong \twedge (\frk_\C/\frr_\C)^\ast$ by 
\[ 
 \lambda(X) = \dfrac{1}{4} \sum_i [X, b_i] \wedge d_i ,\quad X \in \frr_\C
\]
and extended by multiplication to $S(\frr_\C)$. Denote by $\lambda_{\frr_\C}$ the restriction of $\lambda$ to $S(\frr_\C)^{\frr_C}$ and $\lambda_R$ the restriction of $\lambda$ to $S(\frr_\C)^R$.

\begin{lem}\label{wK/wk}
\begin{enumerate}[label={\makebox[3ex][c]{\upshape(\alph*)}}]
\item\label{lemma:wK-wk:item:a}
The group $W_{R_0}=W_{\frr_\C}$ embeds into $W_{R}$ as a normal subgroup.

\item\label{lemma:wK-wk:item:b}
Any element of $R/R_0$ can be represented by an element of $N_{R}(\frt_\C)$; consequently, $R/R_0 \cong N_{R}(\frt_\C)/N_{R_0}(\frt_\C)$.

\item\label{lemma:wK-wk:item:c}
There is a surjective map $R/R_0\to W_R/W_{R_0}$, with kernel $Z_{R}(\frt_\C)/Z_{R_0}(\frt_\C)$. 
\end{enumerate}
\end{lem}

\begin{proof}
\ref{lemma:wK-wk:item:a}\ 
The normalizer $N_{R_0}(\frt_\C)$ is a normal subgroup of $N_{R}(\frt_\C)$ and since
$Z_{R_0}(\frt_\C)=Z_R(\frt_\C)\cap N_{R_0}(\frt_\C)$, $W_{R_0}$ embeds into $W_R$ as a normal subgroup. 

\ref{lemma:wK-wk:item:b}\ 
Let $r\in R$; then $\Ad(r)\frt_\C$ is a Cartan subalgebra of $\frr_\C$. By conjugacy of Cartan subalgebras, there is an $r_0\in R_0$ such that $\Ad(r)\frt_\C=\Ad(r_0)\frt_\C$. Thus $r_0^{-1}r$ normalizes $\frt_\C$, and it represents the same element of $R/R_0$ as $r$. 

Now we consider the obvious map $N_R(\frt_\C)\hookrightarrow R\twoheadrightarrow R/R_0$; it is onto by the above, and the kernel is $N_ R(\frt_\C)\cap R_0=N_{R_0}(\frt_\C)$.

\ref{lemma:wK-wk:item:c}\ 
Inverting the isomorphism obtained in (b), we get
\begin{multline*}
R/R_0\overset{\cong}{\rightarrow} N_R(\frt_C)/N_{R_0}(\frt_\C)\cong \left(N_R(\frt_C)/Z_{R_0}(\frt_\C)\right)/\left(N_{R_0}(\frt_\C)/Z_{R_0}(\frt_\C)\right)\twoheadrightarrow\\ 
\twoheadrightarrow\left(N_R(\frt_\C)/Z_R(\frt_\C)\right)/\left(N_{R_0}(\frt_\C)/Z_{R_0}(\frt)\right)=W_R/W_{R_0}. 
\end{multline*}
\end{proof}

We can now obtain a version of the standard Harish-Chandra isomorphism for disconnected $R$.

\begin{cor}
\label{HC R}
With notation as above, the (usual) Harish-Chandra map $\gamma_0:Z(\frr_\C)=U(\frr_\C)^{R_0}\to S(\frt_\C^\ast)^{W_{\frr_\C}}=S(\frt_\C^\ast)^{W_{R_0}}$ restricts to an isomorphism
\[
\gamma:U(\frr_\C)^{R}\to S(\frt_\C^\ast)^{W_{R}},
\]
similarly 
\[ 
S(\frr_\C)^{R}\cong S(\frt_\C^\ast)^{W_{R}}.
\] 
\end{cor}
\pf 
The group $N_{R}(\frt_\C)$ acts on $U(\frr_\C)^{R_0}$ and $S(\frt_\C^\ast)^{W_{R_0}}$. The map $\gamma_0$ intertwines these actions and descends to an isomorphism 
\[
\gamma: U(\frr_\C)^{R'}\to S(\frt_\C^\ast)^{W_{R}},
\]
where $R'$ is the subgroup of $R$ generated by $R_0$ and $N_{R}(\frt_\C)$. The group $R'$ is however equal to $R$ by Lemma \ref{wK/wk} (b).
\epf

\begin{thm}\label{thm:CartanBorelforDisconnected}
    Let $K$ be a compact connected Lie group, and let $R$ be a symmetric subgroup of $K$, not  necessarily connected then
\begin{equation*}
H_{dR}(\itK/R) \simeq \twedge(\bbP_{\frk_\C/\frr_\C}) \otimes \Bigl( S(\frt_\C^*)^{W_{R}}/ \langle S_+(\frt_\C^*)^{W_{\frk_\C}} \rangle \Bigr).
\end{equation*}
\end{thm}
\begin{proof}
    Theorem \ref{thm:Cartan-Borel} applied to $K/R_0$ states that 
    \begin{equation*}
H_{dR}(\itK/R_0) \simeq \twedge(\bbP_{\frk_\C/\frr_\C}) \otimes \Bigl( S(\frt_\C^*)^{W_{\frr_\C}}/ \langle S_+(\frt_\C^*)^{W_{\frk_\C}} \rangle \Bigr)
\end{equation*}
As shown in \cite{CNP.arxiv2023} and \cite{CGKP.2025}, $\lambda_{\frr_\C}$ gives the characteristic subalgebra
\[
\im \lambda_{\frr_\C} \cong S(\frt_\C^*)^{W_{\frr_\C}}/ \langle S_+(\frt_\C^\ast)^{W_{\frk_\C}} \rangle.
\]
Furthermore, Theorem \ref{thm:Cartan-deRham} applied to both $K/R$ and $K/R_0$ shows that 

\[ H_{dR}(\itK/R) \simeq (\twedge(\frk_\C/\frr_\C)^\ast)^R = ((\twedge(\frk_\C/\frr_\C)^\ast)^{R_0})^R = H_{dR}(K/R_0)^R.\]
Lemma 4.34 in \cite{CGKP.2025} shows that the relative primitives $\bbP_{\frk_\C/\frr_\C}$ are exactly the restriction of the absolute primitives $\bbP_{\frk_\C}$ to the space $\frk_\C/\frr_\C$. The absolute primitives are $K$ invariant and thus $R$ invariant. Since $\frr_\C$ is a $R$ submodule, then restriction to $\frk_\C/\frr_\C$ is an $R$ module homomorphism. Hence the relative primitives  $\bbP_{\frk_\C/\frr_\C}$ and thus $\twedge(\bbP_{\frk_\C/\frr_\C})$, are $R$ invariant. Therefore
\[
H_{dR}(\itK/R) \simeq  H_{dR}(K/R_0)^R \simeq \twedge(\bbP_{\frk_\C/\frr_\C}) \otimes  (\im \lambda_{\frr_\C})^R.
\]

We are left to understand the $R$ invariants of $\im \lambda_{\frr_\C}$. In fact $H_{dR}(K/R_0)$ is already $R_0$ invariant, thus we need to take $R/R_0$ invariance, which under the Harish-Chandra isomorphism is given by $W_R/W_{\frr_\C}$ invariance.

By Corollary \ref{cor:R-invariants.in.wedge.n} we know that under the symmetric analogue of the Harish-Chandra isomorphism that the $R$ invariants of $S(\frr_\C)$ correspond to the $W_R$ invariants of $S(\frt_\C^\ast)$. We have the following exact sequence of $R$ modules on the top row and $W_R$ modules on the bottom row:
\begin{equation*}
\begin{tikzcd}
    & & S(\frr_\C)^{\frr_\C} \arrow[d, "\cong"] \arrow[r,"\lambda_{\frr_\C}"]  & \im \lambda_{\frr_\C} \arrow[r] & 0\\
0  \arrow[r] & \langle S_+(\frt_\C^*)^{W_{\frk_\C}}\rangle \arrow[r, hook] &  S(\frt_\C^*)^{W_{\frr_\C}}    & 
\end{tikzcd}
\end{equation*}
Knapp \cite[Main Theorem]{Knapp.1975} shows that  $W_{\frk_\C} = N_K(\frt_\C) /Z_K(\frt_\C)$, thus $W_{\frk_\C} \supset W_R$ and $S_+(\frt_\C^*)^{W_{\frk_\C}}$ is $W_R$ invariant. To distinguish between two ideals, let $I_+({W_{\frr_\C}})$ and $I_+({W_R})$ denote the ideal generated by $S_+(\frt_\C^\ast)^{W_{\frk_\C}}  $ in $S(\frt_\C^\ast)^{W_{\frr_\C}}$ and $S(\frt_\C^\ast)^{W_R}$ respectively.

We claim that 
\[
\left(I_+({W_{\frr_\C}}) \right)^{W_R} =I_+({W_R}).
\]
Consider a general element $x = \sum a_i b_i $ in $I_+({W_{\frr_\C}})$ with $a_i \in S(\frt_\C^\ast)^{W_{\frr_\C}}$ and $b_i \in  S_+(\frt_\C^\ast)^{W_{\frk_\C}}$. Since every $b_i$ is $W_R$ invariant the forward containment is clear. Suppose that $x$ is $W_R$ invariant, then the Reynolds operator $\pi^{W_R}(x) = \frac{1}{|W_R|} \sum_{w \in W_R} w (x)$ fixes $x$. Thus 
\[
x = \pi^{W_R}(x) = \pi^{W_R} (\sum a_i b_i ) = \sum \pi^{W_R} (a_i) b_i,
\] with the last equality following from the fact that $\pi^{W_R}$ is $S(\frt_\C^\ast)^{W_R}$ linear.  Since $\pi^{W_R}(a_i) \in S(\frt_\C^\ast)^{W_R}$ then $x \in  I_+({W_R})$. Thus proving our claim.  

Taking $R$ and $W_R$ invariance of the short exact sequence for $\lambda_{\frr_\C}$ we then find 
\begin{equation*}
\begin{tikzcd}
      & & S(\frr_\C)^R  \arrow[r,"\lambda_R"] \arrow[d,"\cong"]  & \im (\lambda_{\frr_\C})^R \arrow[r] & 0 \\
       0 \arrow[r] & I_+(W_R) \arrow[r] & S(\frt_\C^\ast)^{W_R}    
\end{tikzcd}
\end{equation*}
which gives the description of $(\im \lambda_{\frr_\C})^R$ as required.

\end{proof}

\begin{rem}
    It is possible to prove Theorem \ref{thm:CartanBorelforDisconnected} for a general closed subgroup $R \subset K$, under the assumption that $(\frk_\C,\frr_\C)$ is a Cartan pair. This proof heavily uses the results of \cite{CCC.III.1976}. Since we only require this theorem for symmetric pairs we omit the details.
\end{rem}

\section{Examples}\label{section:examples}

\subsection{Case of $ \itG = \UU(n, n) $}

Let us consider $ \itG = \UU(n, n) $ over $ \C $ (indefinite unitary group). 
From the table in \cite[Table{1}]{CNP.arxiv2023}, we know 
$ \itP = \stab(\C^n) $ (the stabilizer of a Lagrangian subspace $ \C^n $), 
$ \itL = \GL_n(\C) $, $ \itK = \UU(n)^2 $ and $ R = \Delta \UU(n) $ (diagonal subgroup).  
Thus $ (\itG_\C, \itL_\C) = (\GL_{2n}(\C), \GL_n(\C)^2) $ is a symmetric pair, as well as 
$ (\itG, \itL) = (\UU(n, n), \GL_n(\C))$, $ (\itK, R) = (\UU(n)^2, \Delta \UU(n)) $ over $ \R $.  

It is well known that 
\begin{equation}\label{eq:decomposition.wedge.Matn}
\twedge(\frn_\C) \simeq \twedge(\Mat_n(\C)) \simeq \bigoplus_{\lambda} V_{\lambda} \boxtimes V_{\transpose{\lambda}} 
\end{equation}
as an $ \itL_\C $-module, where $ \lambda $ moves over the Young diagrams contained in the $ n \times n $-square shape and $\transpose{\lambda}$ is the transpose of $\lambda$. 

$ \itK $ and $ R $ do not have the same rank; the ranks are $ 2n $ and $ n $.  
So we use the diagonal torus $ T \subset R $ to define the Weyl groups, and 
we get $ W_{\frk_\C} = W_{\frr_\C} = S_n $, the symmetric group of order $ n $.  
Thus the dimension of the de Rham cohomology is given by 
\begin{equation*}
\dim H_{dR}(\itG/\itP) = \dim H_{dR}(\itK/R) = 2^{\rank \frk_\C - \rank \frr_\C} = 2^n.  
\end{equation*}
Notice that $ \itG/\itP $ is the Grassmannian of Hermitian Lagrangian subspaces in $ \C^{n|n} $ and 
$ \itK/R \simeq \UU(n) $, denoted by $ \HLGr(\C^{n|n}) $.  
To see this, we just need to note two facts: 
$ G = \UU(n, n) $ acts transitively on the space of Hermitian Lagrangian subspaces and $ P $ stabilizes the base Lagrangian $ \C^n \oplus \{ 0 \} $.

If we take $ R $-invariants in the formula \eqref{eq:decomposition.wedge.Matn}, we get
\begin{equation*}
\twedge(\frn_\C)^R \simeq \twedge(\Mat_n(\C))^{\Delta \GL_n(\C)} \simeq \bigoplus_{\lambda} (V_{\lambda} \boxtimes V_{\transpose{\lambda}})^{\Delta \GL_n(\C)}.
\end{equation*}
In the last formula, since $ g \in \Delta \GL_n(\C) $ acts by $ (g, \transpose{g}^{-1}) $, 
the term for $ \lambda $ survives only for $ \lambda $ satisfying $ \lambda = \transpose{\lambda} $, i.e., 
Young diagrams in $ n \times n $-square, which are symmetric along the diagonal. 
Each of such $ \lambda $ contributes $ 1 $ dimension to the invariants. 
Thus we get the formula
\begin{equation*}
|\{\lambda \mid \lambda \subset [n]\times [n], \lambda = \transpose{\lambda} \}| = 2^n.
\end{equation*}
(one can deduce this from combinatorial arguments, but here we get it from the dimension formula of cohomologies).  
The Poincar\'{e} series is given by 
\begin{equation*}
P(\operatorname{HLGr}(\C^{n|n}), q) = P(\UU(n), q) = \prod_{k = 0}^n (1 + q^{2 k - 1}).
\end{equation*}
%
To see this, a symmetric Young diagram $ \lambda \subset [n]\times [n] $ is the ``nest'' of symmetric hooks 
$ (k, 1^{k- 1}) $ for $ 1 \leq k \leq n $ of size $ 2 k - 1 $.  
The degree should be counted as the number of cells (boxes) in the symmetric Young diagram.
The nest allows any combination so that there are choices of $ q^{2 k - 1} $ for each $ k $.

In this case, the formula above also gives the Poincare series 
$ P(\twedge(\bbP_{\frr_\C}), q) $ 
of the exterior algebra of the primitives. 

\subsection{Case of $ G = \Sp_{2n}(\R) $}\label{subsec:example.sp2nR}

Let us consider the case where $ \itG = \Sp_{2n}(\R) $ and $ \itP $ is the Siegel parabolic.  
In this case, $ \itL = \GL_n(\R) $ is a Levi component of $ \itP $, and the nilpotent radical is given by $ \frn = \Sym_n(\R) $.  
A maximal compact subgroup of $ \itG $ is $ \itK = \UU(n) $ and $ R = \itL \cap \itK = \OO(n) $.  Thus we have an isomorphism $ \itG/ \itP 
= \Sp_{2n}(\R)/ P \simeq \UU(n)/ \OO(n) = \itK/ R $.  
Note that $ \itL $ and $ R $ are not connected, but have two connected components, thus the classical Cartan-Borel theorem does not apply directly however we can use Theorem \ref{thm:CartanBorelforDisconnected}.

First, we examine the $ \itL_{\C} $-module structure on $ \twedge(\frn_\C) $ and determine $ R $-invariants as well as $ R_0 $-invariants.  

The adjoint action of the Levi subgroup $ \itL_\C = \GL_n(\C) $ on $ \frn_\C = \Sym_n(\C) $ is the usual one, i.e., 
$ g \in \GL_n(\C) $ acts on $ z \in \Sym_n(\C) $ as $ g \cdot z = g z \transpose{g} $.  This representation is irreducible.  
Let us explain which irreducible representations $ V_{\lambda} $ of $ \GL_n(\C) $ appear in $ \twedge(\Sym_n(\C)) $.  

Let us consider a hook $ h_k = (k + 1, 1^{k - 1}) $ for $ k \geq 1 $.  
Then $ \lambda $ must be a Young diagram obtained by nesting $ h_k $'s according to \cite[Theorem 4.4.2]{Howe.SchurLecture.1995}.  

(Nesting means to assemble the hooks in such a way that the vertices of the hooks line along the diagonals, and the resulting assembly makes a diagram.   In the present case,  one can only use $ h_k $ at most once for each $ k $.)
 
We reinterpret Howe's result in such a way where the description is more convenient for us.    
In Frobenius notation for Young diagrams, the hook is denoted by $ h_k = [(k); (k-1)] $.  
For example the nest of $ h_4 $ and $ h_2 $ gives $ \lambda = (5,4,2,1) $, which is written as 
$ \lambda = [(4,2); (3,1)] $ in Frobenius notation.  (For Frobenius notation, see \cite[Chapter~I, \S1, p.~3]{Macdonald.1995}.)
A partition $ \lambda $ obtained in that way always has even size because $ |h_k| = 2 k $ is even. 

\begin{lem}[{\cite[Theorem 4.4.2]{Howe.SchurLecture.1995}}]\label{lem:Sp2n.wedgeSym_n}
An irreducible representation $ V_{\lambda} $ of $ \GL_n(\C) $ appears in $ \twedge(\Sym_n(\C)) $ if and only if 
$ \lambda = [\mu; \eta] $ satisfies the conditions:
\begin{enumerate}[label={\upshape(\arabic*)}]
\item\label{lem:Sp2n.wedgeSym_n:item:1}
$ \mu $ is a strict partition (partition with mutually distinct parts) of length 
$ \ell := \ell(\mu) \leq n $ and $ \mu_1 \leq n $, 
\item\label{lem:Sp2n.wedgeSym_n:item:2}
$ \eta = \mu - (1^{\ell}) $, i.e., $\eta$ is $\mu$ with first column deleted.
\end{enumerate}

If $ V_{\lambda} $ appears in $ \twedge(\Sym_n(\C)) $, it is contained in the $ p $-th graded piece $ \twedge^p(\Sym_n(\C)) $ with $ p = |\lambda|/2 = |\mu| $, with multiplicity one. 
\end{lem}

\begin{rem}
The conditions $ \ell(\mu) \leq n $ and $ \mu_1 \leq n $ ensure that $ \lambda = [\mu;\eta] $ is contained in the $ n \times (n + 1) $ rectangular shape.
\end{rem}

\begin{lem}\label{lem:V.with.On.SOn.invariants}
Let $ V_{\lambda} $ be a representation of $ \GL_n(\C) $ which appears in $ \twedge(\Sym_n(\C)) $, and write 
$ \lambda = [\mu;\eta] $ as in Lemma~\ref{lem:Sp2n.wedgeSym_n}.  
\begin{enumerate}[label={\upshape(\arabic*)}]
\item\label{lem:V.with.On.SOn.invariants:item:On}
$ V_{\lambda}^{\OO_n(\C)} \neq \{ 0 \} $ (and consequently is isomorphic to $ \C $) if and only if 
for any $ j \leq \ell(\mu)/2 $, 
$ (\mu_{2 j - 1}, \mu_{2 j}) = (2 k + 1, 2 k) $ holds for some $ k \geq 0 $, 
i.e., consecutive two parts of $ \mu $ beginning with odd position is a pair of an odd number and an even number (in this order) with difference one.
\item\label{lem:V.with.On.SOn.invariants:item:SOn:n.even}
Let us assume that $ n $ is even.  Then $ V_{\lambda}^{\SO_n(\C)} \neq \{ 0 \} $ (and consequently is isomorphic to $ \C $) if and only if 
$ \lambda $ satisfies the former condition \ref{lem:V.with.On.SOn.invariants:item:On}, 
or the largest hook of the nest in $ \lambda $ is $ h_n $, and the union of the rest of the nest satisfies 
the condition in \ref{lem:V.with.On.SOn.invariants:item:On}.
\item\label{lem:V.with.On.SOn.invariants:item:SOn:n.odd}
Let us assume that $ n $ is odd.  
Then $ V_{\lambda}^{\SO_n(\C)} \neq \{ 0 \} $ (and consequently is isomorphic to $ \C $) if and only if 
$ V_{\lambda}^{\OO_n(\C)} \neq \{ 0 \} $ (hence satisfying the condition in \ref{lem:V.with.On.SOn.invariants:item:On}).
\end{enumerate}
\end{lem}

\begin{proof}
We know (see \cite{Howe.SchurLecture.1995}) $ V_{\lambda}^{\OO(n)} \simeq \C $ if and only if $ \lambda = 2 \alpha $ for some partition $ \alpha $, i.e., 
all the parts of $ \lambda $ are even ($ \lambda $ is called an even partition).  
Similarly $ V_{\lambda}^{\SO(n)} \simeq \C $ if and only if all the parts of $ \lambda $ have the same parity (i.e., if and only if $ \lambda $ is an even partition or an odd partition).  
Note that if $ \lambda $ is an odd partition, then the $n$ parts of $ \lambda $ do not contain zero. 

We can translate the conditions on $ \lambda $ in Lemma~\ref{lem:Sp2n.wedgeSym_n} combinatorially to get \ref{lem:V.with.On.SOn.invariants:item:On}. 
For \ref{lem:V.with.On.SOn.invariants:item:SOn:n.even}, if we take off the largest hook $ h_n $ from an odd partition (this must appear since it is odd), the rest is an even partition. 
So we can use \ref{lem:V.with.On.SOn.invariants:item:On}.  
However, if $ n $ is odd as in \ref{lem:V.with.On.SOn.invariants:item:SOn:n.odd}, we cannot add $ h_n = (n + 1, 1^{n - 1}) $ to assemble an odd partition since $ n + 1 $ is even.  
Thus, $ \lambda $ must be an even partition. 
\end{proof}

We denote the set of partitions $ \lambda $ appearing in the decomposition of $ \twedge(\Sym_n(\C)) $ by $ \Xi_n $.  
Namely, 
\begin{equation*}
\Xi_n = \{ \lambda = [\mu; \eta] \mid \text{$ \mu $ and $ \eta $ satisfy the conditions in Lemma \ref{lem:Sp2n.wedgeSym_n}} \}.
\end{equation*}
Since $ \eta $ is completely determined by $ \mu $, 
$ \Xi_n $ is in bijection with the strict partitions $ \mu $ of size $ \leq \frac{1}{2} n (n + 1) $ with parts at most $n$.  

Consequently, we find
\begin{align*}
&
\twedge(\frn_\C)^R \simeq \twedge(\Sym_n(\C))^{\OO(n)} \simeq \bigoplus_{\lambda \in \Xi_n} V_{\lambda}^{\OO(n)}, 
\\
&
\twedge(\frn_\C)^{\frr_\C} \simeq \twedge(\Sym_n(\C))^{\SO(n)} \simeq \bigoplus_{\lambda \in \Xi_n} V_{\lambda}^{\SO(n)}.
\end{align*}
Let us calculate the Poincar\'{e} polynomials of the cohomologies.
To do so, we divide the cases.

\subsubsection{}
Let us assume $ n = 2m $ is even.

We take a Cartan subalgebra $ \frt_\C \subset \frr_\C $ and consider 
the Weyl groups $ W_{\frk_\C} $ and $ W_{\frr_\C} $ acting on this Cartan subalgebra.  
Note that $ \frt_\C $ is \emph{not} a Cartan subalgebra in $ \frk_\C $ since 
$ \rank \frk_\C = n = 2 m $ while $ \rank \frr_\C = m $.  
After some reflection, we find $ W_{\frk_\C} = W_{C_m} $ and $ W_{\frr_\C} = W_{D_m} $, so that 
$  |W_{\frk_\C}/W_{\frr_\C}| = 2 $.  
While we get $ |W_{\itK}/W_R| = 1 $ \cite[Subsection 3.8]{CNP.arxiv2023}.

Then, we get 
\begin{align*}
\dim H_{dR}(\itG/\itP) &= \dim H_{dR}(\itK/R) = 2^{\rank \frk_\C - \rank \frr_\C} |W_{\itK}/W_R| = 2^m, 
\\
\dim H_{dR}(\itK/R_0) &= 2^{\rank \frk_\C - \rank \frr_\C} |W_{\frk_\C}/W_{\frr_\C}| = 2^{m + 1}.
\end{align*}
Note that $ W_K = W_{\frk_\C} $ which contains $ W_{\frr_\C} $ with index $ 2 $.

\begin{thm}\label{thm:Poincare.series.Sp2n:n.even}
Assume $ n = 2 m $ is even.  
The Poincar\'{e} series of the symmetric spaces are given as follows.
\begin{align*}
P(\operatorname{LGr}(\R^{2n}), q) = P(\UU(n)/\OO(n), q) 
&= \prod_{k = 0}^{\lfloor (n - 1)/ 2 \rfloor} (1 + q^{4 k + 1}) 
= \prod_{k = 0}^{m - 1} (1 + q^{4 k + 1}) 
\\
P(\UU(n)/\SO(n), q) 
&= \prod_{k = 0}^{\lfloor (n - 1)/ 2 \rfloor} (1 + q^{4 k + 1}) + q^n \prod_{k = 0}^{\lfloor (n - 2)/ 2 \rfloor} (1 + q^{4 k + 1}) 
\\
&= (1 + q^n) \prod_{k = 0}^{m - 1} (1 + q^{4 k + 1}) 
\end{align*}
\end{thm}

\begin{proof}
We reinterpret the partitions $\lambda$, which appear in Lemma~\ref{lem:V.with.On.SOn.invariants} \ref{lem:V.with.On.SOn.invariants:item:On}.  
The Young diagram $ \lambda $ is a nested union of 
the partitions $ [ (2 k + 1, 2 k) ;  (2 k, 2 k - 1) ] $ ($ k \geq 1 $, it becomes $ [(1); (0)] $ if $ k = 0 $). 
The size of $ [ (2 k + 1, 2 k) ;  (2 k, 2 k - 1) ] $ is $ 8 k + 2 $, and these sizes are all different.  
So we have a strict partition of parts contained in $ \{ 8 k + 2 : k \geq 0 \} $.  The maximal size of $ k $ satisfies $ 2 k + 1 \leq n $ because $ \lambda $ is contained in the $ n \times (n + 1) $ rectangular shape.  
Thus it is a strict partition of a number less than or equal to $ n^2 = 4 m^2 $.  
But this condition is automatically satisfied if we consider strict partitions with parts satisfying the above conditions. 
Since the grading is given by $ |\lambda|/2 $, for the Poincar\'{e} series, we use the strict partitions consisting of $ 4 k + 1 = (8 k + 2)/2 $, which gives the formula in the theorem.  

The case of $ R_0 $ is similar.
\end{proof}

\subsubsection{}
Let us assume $ n = 2m + 1 $ is odd.

As in the case of $ n $ is even, we take a Cartan subalgebra $ \frt_\C \subset \frr_\C $ and consider 
the Weyl groups $ W_{\frk_\C} $ and $ W_{\frr_\C} $ acting on this Cartan subalgebra.  
Now note that $ \rank \frk_\C = n = 2 m + 1 $ while $ \rank \frr_\C = m $.  
After some reflection, we find $ W_{\frk_\C} = W_{C_m} $ and $ W_{\frr_\C} = W_{B_m} $, so that 
$  |W_{\frk_\C}/W_{\frr_\C}| = 1 $.    
We also get $ |W_{\itK}/W_R| = 1 $ \cite[Subsection 3.8]{CNP.arxiv2023}.
Thus we have 
\begin{align*}
\dim H_{dR}(\itG/\itP) &= \dim H_{dR}(\itK/R) = 2^{\rank \frk_\C - \rank \frr_\C} |W_{\itK}/W_R| = 2^{m + 1}, 
\\
\dim H_{dR}(\itK/R_0) &= 2^{\rank \frk_\C - \rank \frr_\C} |W_{\frk_\C}/W_{\frr_\C}| = 2^{m + 1}.
\end{align*}
Their dimensions are equal.

\begin{thm}\label{thm:Poincare.series.Sp2n:n.odd}
Assume $ n = 2 m + 1 $ is odd.  
The Poincar\'{e} series of the symmetric spaces are given as follows.
\begin{align*}
P(\operatorname{LGr}(\R^{2m + 1}), q) = P(\UU(n)/\OO(n), q) 
&= \prod_{k = 0}^{\lfloor (n - 1)/ 2 \rfloor} (1 + q^{4 k + 1}) 
= \prod_{k = 0}^{m} (1 + q^{4 k + 1}) 
\\
P(\UU(n)/\SO(n), q) 
&= \prod_{k = 0}^{\lfloor (n - 1)/ 2 \rfloor} (1 + q^{4 k + 1}) 
= \prod_{k = 0}^{m} (1 + q^{4 k + 1}) 
\end{align*}
In particular, they have the same Poincar\'{e} series.
\end{thm}

\begin{proof}
For $ R $-invariants, we have the same reasoning as in the proof of Theorem~\ref{thm:Poincare.series.Sp2n:n.even}.
Lemma~\ref{lem:V.with.On.SOn.invariants} tells us that the $ R_0 $ invariants are the same as the $ R $-invariants.
\end{proof}

\subsection{The case of $ G = \GL_n(\R) $}
\newcommand{\mbfq}{\boldsymbol{q}}

This case is corresponding to Casian-Kodama \cite{Casian.Kodama.2013}.
We take $ P = P_{(p, q)} \subset G $, $ L = \GL_p(\R) \times \GL_q(\R) $, 
$ K = \OO(n) $, $ R = \OO(p) \times \OO(q) $, where $ p + q = n $.  
Note that $ G $ is not connected as a Lie group, but as an algebraic group, it is connected (in Zariski topology).
We denote $ K_\C = \OO_n(\C) $ and $ R_\C = \OO_p(\C) \times \OO_q(\C) $. 

Note that 
$ \OO(n)/\bigl( \OO(p) \times \OO(q) \bigr) \simeq \SO(n)/S\bigl( \OO(p) \times \OO(q) \bigr) $ is isomorphic to the Grassmannian $ \Gr_p(\R^{p + q}) $ 
and $ \SO(n)/\bigl( \SO(p) \times \SO(q) \bigr) $ is its double cover. In particular we can apply Theorem \ref{thm:CartanBorelforDisconnected}.

The nilpotent radical of $ \frP = \Lie P $ is $ \frn = \Mat_{p, q}(\R) $ and after complexification we get 
\begin{equation*}
\twedge (\frn_\C) = \twedge (\Mat_{p,q}(\C)) = \sum_{\lambda} V_{\lambda}^{(p)} \otimes V_{\transpose{\lambda}}^{(q)} \quad (\text{as an $ L_\C $-module}), 
\end{equation*}
where $ \lambda $ runs over the Young diagrams contained inside the $ p \times q $ rectangular shape, 
and $V_{\lambda}^{(p)} $ means that it is a representation of $ \GL_p(\C) $.  
We see that $ (V_{\lambda}^{(p)} \otimes V_{\transpose{\lambda}}^{(q)})^{R_{\C}} \neq 0 $ if and only if 
$ \lambda $ is a Young diagram consisting of $ 2 \times 2 $ squares (instead of single boxes).  
Put $ a = \bigl\lfloor \dfrac{p}{2} \bigr\rfloor, b = \bigl\lfloor \dfrac{q}{2} \bigr\rfloor $ and $ m = \bigl\lfloor \dfrac{n}{2} \bigr\rfloor $.  

To state the results clearly, and to compare the statements of the Cartan-de Rham theorem and the Cartan-Borel theorem, we divide the cases.

\subsubsection{$ (p, q) \equiv (0, 0) \pmod{2} $}
\begin{align*}
& \rank \frk = \rank \frr & & m = a + b 
\\
& W_K = B_m & & W_R = B_a \times B_b 
\\
& W_{\frk_\C} = D_m & & W_{\frr_\C} = D_a \times D_b
\end{align*}
In this case, $ \OO(p) \times \OO(q) $-invariants and $ \SO(p) \times \SO(q) $-invariants in $ \twedge (\frn_\C) $ are the same.
Thus the
Cartan-de Rham Theorem gives the Poincar\'{e} polynomial:
\begin{equation*}
P(\SO(n)/S\bigl( \OO(p) \times \OO(q) \bigr), \mbfq) = \begin{bmatrix} a + b \\ a \end{bmatrix}_{\mbfq^4}.
\end{equation*}
Here $ \begin{bmatrix} m \\ k \end{bmatrix}_{\mbfq} $ denotes the $ q $-binomial coefficient (also called Gaussian binomial coefficient) defined as
\begin{equation*}
\begin{bmatrix}
m\\ k
\end{bmatrix}_{\mbfq}
= \frac{[m]_{\mbfq} !}{[k]_{\mbfq} !\, [m-k]_{\mbfq} !}
= \frac{(1-\mbfq^m)(1-\mbfq^{m-1})\cdots(1-\mbfq^{m-k+1})}
     {(1-\mbfq)(1-\mbfq^2)\cdots(1-\mbfq^k)}
= \prod_{i=1}^{k}
\frac{1-\mbfq^{m-k+i}}{1-\mbfq^i}.
\end{equation*}
It is well known that this $ q $-binomial coefficient gives a generating function of the enumeration of Young diagrams inside the rectangular shape $ m \times k $.
See \cite[\S{1.7}]{Stanley.vol1.2012}.
Cartan-Borel theorem applied to $ K/ R_0 $ shows that:
\begin{equation*}
\dim H_{dR}(\SO(n)/\bigl( \SO(p) \times \SO(q) \bigr)) = 2 \binom{a + b}{a}. 
\end{equation*}

\subsubsection{$ (p, q) \equiv (1, 0) \pmod{2} $}
\begin{align*}
& \rank \frk = \rank \frr & &  m = a + b 
\\
& W_K = B_m & &  W_R = B_a \times B_b 
\\
& W_{\frk_\C} = B_m & &  W_{\frr_\C} = B_a \times D_b
\end{align*}
Again, in this case, $ \OO(p) \times \OO(q) $-invariants and $ \SO(p) \times \SO(q) $ invariants in $ \twedge (\frn_\C) $ are the same.
Thus 
Cartan-de Rham Theorem tells:
\begin{equation*}
P(\SO(n)/S\bigl( \OO(p) \times \OO(q) \bigr), \mbfq) = \begin{bmatrix} a + b \\ a \end{bmatrix}_{\mbfq^4}.
\end{equation*}
Cartan-Borel theorem tells:
\begin{equation*}
\dim H_{dR}(\SO(n)/\bigl( \SO(p) \times \SO(q) \bigr)) = 2 \binom{a + b}{a}.
\end{equation*}
Again $ 2 $ appears because it is a double cover.

\subsubsection{$ (p, q) \equiv (1, 1) \pmod{2} $}

In this case, we take a Cartan subalgebra $ \frt $ from $ \frr $, which is strictly contained in the Cartan subalgebra of $ \lie{k} $.
\begin{align*}
& \rank \frk = \rank \frr + 1 & &  m = a + b + 1 
\\
& W_K = B_{a + b} & &  W_R = B_a \times B_b 
\\
& W_{\frk_\C} = B_{a + b} & &  W_{\frr_\C} = B_a \times B_b
\end{align*}
In this case, $ \OO(p) \times \OO(q) $-invariants and $ \SO(p) \times \SO(q) $ invariants in $ \twedge (\frn_\C) $ are different.  
Cartan-de Rham Theorem implies:
\begin{align*}
P(\SO(n)/S\bigl( \OO(p) \times \OO(q) \bigr), \mbfq) 
&
= (1 + \mbfq^{p + q - 1}) \begin{bmatrix} a + b \\ a \end{bmatrix}_{\mbfq^4}.
\end{align*}
Cartan-Borel theorem implies:
\begin{equation*}
\dim H_{dR}(\SO(n)/\bigl( \SO(p) \times \SO(q) \bigr)) = 2 \binom{a + b}{a}.
\end{equation*}
In this case the de Rham cohomologies of 
$ \SO(n)/S\bigl( \OO(p) \times \OO(q) \bigr) $ and $ \SO(n)/\bigl( \SO(p) \times \SO(q) \bigr) $ are the same.

\subsection{Case of $ G = \SO^*(4n) $}

We consider the case where $ \itG = \SO^*(4n) $ and $ \itP $ is the Siegel parabolic with 
a Levi component $ \itL = \GL_n(\bbH) $ and the nilpotent radical $ \frn = \Her_n(\bbH) $, the $n\times n$ Hermitian matrices over the quaternions.  
A maximal compact subgroup of $ \itG $ is $ \itK = \UU(2n) $ and $ R = \itL \cap \itK = \USp(2n) $, the unitary symplectic group of size $ 2 n $ (of rank $ n $), namely 
$ \USp(2n) = \UU(2 n) \cap \Sp_{2 n}(\C) $.  We have an isomorphism $ \itG/ \itP = \SO^*(4n)/ P \simeq \UU(2n)/ \USp(2n) = \itK/ R $.  
Both $ \itL $ and $ R $ are connected in this case.

First, we examine the $ \itL_{\C} $-module structure on $ \twedge(\frn_\C) $ and determine $ R $-invariants.  

Let us denote the space of alternating matrices of size $ m $ by $ \Alt_m(\C) $.  
In the standard realization, the nilpotent radical $ \frn_\C $ of $ \frP_\C $ can be identified with $ \Alt_{2n}(\C) $ and  
the adjoint action of the Levi subgroup $ \itL_\C = \GL_{2n}(\C) $ on $ \frn_\C = \Alt_{2n}(\C) $ is unimodular, i.e., 
$ g \in \GL_{2n}(\C) $ acts on $ z \in \Alt_{2n}(\C) $ as $ g \cdot z = g z \transpose{g} $. 
This representation is irreducible.  

Let us explain which irreducible representations $ V_{\lambda} $ of $ \GL_{N}(\C) $ appear in $ \twedge(\Alt_{N}(\C)) $ for general $ N $.  
Define a hook $ \widetilde{h}_k = (k, 1^{k}) $ for $ k \geq 1 $ (the transpose of the hook $ h_k $ in \S\ref{subsec:example.sp2nR}).  
Then $ \lambda $ must be a Young diagram obtained by nesting $ \widetilde{h}_k $'s (it must make a diagram, so that one can only use $ \widetilde{h}_k $ at most once for each $ k $) according to \cite[Theorem 4.4.4]{Howe.SchurLecture.1995}.  
We reinterpret Howe's result in such a way that the description becomes more convenient for us.    
In Frobenius notation for Young diagrams, the hook is denoted by $ \widetilde{h}_k = [(k - 1); k] $.  
For example nest of $ \widetilde{h}_4 $ and $ \widetilde{h}_2 $ gives $ \lambda = (4,3,2,2,1) $, which is written as 
$ \lambda = [(3,1); (4,2)] $ in Frobenius notation (this example shows everything is transposed comparing with \S\ref{subsec:example.sp2nR}).  
A partition $ \lambda $ obtained in that way always has even size because $ |\widetilde{h}_k| = 2 k $ is even.

\begin{lem}[{\cite[Theorem 4.4.4]{Howe.SchurLecture.1995}}]\label{lem:SO4n.wedgeAlt}
An irreducible representation $ V_{\lambda} $ of $ \GL_N(\C) $ appears in $ \twedge(\Alt_N(\C)) $ if and only if 
$ \lambda = [\mu; \eta] $ satisfies the conditions 
\begin{enumerate}[label={\upshape(\arabic*)}]
\item\label{lem:SO4n.wedgeAlt:item:1}
$ \eta $ is a strict partition (partition with mutually distinct parts) of length 
$ \ell := \ell(\eta) < N $ and $ \eta_1 < N $, 
\item\label{lem:SO4n.wedgeAlt:item:2}
$ \mu = \eta - (1^{\ell}) $.
\end{enumerate}
If the representation $ V_{\lambda} $ appears in $ \twedge(\Alt_N(\C)) $, it is contained in the $ p $-th graded piece $ \twedge^p(\Alt_N(\C)) $ with $ p = |\lambda|/2 = |\eta| $,  with multiplicity one. 
\end{lem}

\begin{rem}
The conditions $ \ell(\eta) < N $ and $ \eta_1 < N $ ensure that $ \lambda = [\mu;\eta] $ is contained in the $ N \times (N - 1) $ rectangular shape.
\end{rem}

Now we consider the case where $ N = 2n $ is even.

\begin{lem}\label{lem:V.with.USp.invariants}
Let $ V_{\lambda} $ be a representation of $ \GL_{2n}(\C) $ which appears in $ \twedge(\Alt_{2n}(\C)) $, and write it 
$ \lambda = [\mu;\eta] $ as in Lemma~\ref{lem:SO4n.wedgeAlt}.  
$ V_{\lambda}^{\Sp_{2n}(\C)} \neq \{ 0 \} $ (and consequently is isomorphic to $ \C $) if and only if 
for any $ j \leq \ell(\eta)/2 $, 
$ (\eta_{2 j - 1}, \eta_{2 j}) = (2 k + 1, 2 k) $ holds for some $ k \geq 0 $, 
i.e., consecutive two parts of $ \eta $ beginning with odd position is a pair of an odd number and an even number (in this order) with difference one.
\end{lem}

\begin{proof}
We know $ V_{\lambda}^{\Sp_{2n}(\C)} \simeq \C $ if and only if $ \transpose{\lambda} = 2 \alpha $ for some partition $ \alpha $, i.e., 
$ \lambda $ is the transpose of an even partition.  We often denote $ \lambda = \alpha^2= (\alpha_1,\alpha_1,\alpha_2,\alpha_2,\dots) $ when $ \transpose{\lambda} = 2 \alpha = (2 \alpha_1, 2 \alpha_2, \dots) $.  
Thus we can translate the proof of Lemma~\ref{lem:V.with.On.SOn.invariants} taking the transpose all over.
\end{proof}

We denote the set of partitions $ \lambda $ appearing in the decomposition of $ \twedge(\Alt_N(\C)) $ by $ \widetilde{\Xi}_N $.  
Namely, 
\begin{equation*}
\widetilde{\Xi}_N = \{ \lambda = [\mu; \eta] \mid \text{$ \mu $ and $ \eta $ satisfy the conditions in Lemma \ref{lem:SO4n.wedgeAlt}} \}.
\end{equation*}
Since $\mu$ is completely determined by $ \eta $, 
$ \widetilde{\Xi}_N $ is in bijection with the strict partitions $ \eta $ of size $ \leq \frac{1}{2} N (N - 1) $ with parts at most $ N - 1 $.  Thus we get
\begin{align*}
&
\twedge(\frn_\C)^R = \twedge(\frn_\C)^{\frr_\C} 
\simeq \twedge(\Alt_{2n}(\C))^{\Sp_{2n}(\C)} \simeq \bigoplus_{\lambda \in \widetilde{\Xi}_N} V_{\lambda}^{\Sp_{2n}(\C)} .
\end{align*}
The Poincar\'{e} polynomials of the cohomologies are the same as in \S\ref{subsec:example.sp2nR} for $ n $ even.

We take a Cartan subalgebra $ \frt_\C \subset \frr_\C $ and consider 
the Weyl groups $ W_{\frk_\C} $ and $ W_{\frr_\C} $ acting on this Cartan subalgebra.  
Since $ \rank \frk_\C = 2 n $ and $ \rank \frr_\C = n $, 
$ \frt_\C $ is not a Cartan subalgebra of $ \frk_\C $.  
We find both of $ W_{\frk_\C} $ and $ W_{\frr_\C} $ are isomorphic to $ W_{C_n} $ so that $  |W_{\frk_\C}/W_{\frr_\C}| = 1 $ and $ |W_{\itK}/W_R| = 1 $.  
Thus, Cartan-Borel theorem tells 
\begin{align*}
\dim H_{dR}(\itG/\itP) &= \dim H_{dR}(\itK/R) = 2^{\rank \frk_\C - \rank \frr_\C} |W_{\itK}/W_R| = 2^n. 
\end{align*}
On the other hand, from Theorem~\ref{thm:dim.Poincare.series.of.cohomology}, we get the Poincar\'{e} series.

\begin{thm}\label{thm:Poincare.series.SOstar4n}
The Poincar\'{e} series of the symmetric space $ \UU(2n)/\USp(2n) $ is given by
\begin{align*}
P(\operatorname{HGr}(\bbH^{n}), q) = P(\UU(2n)/\USp(2n), q) 
&= \prod_{k = 0}^{n - 1} (1 + q^{4 k + 1}) 
\end{align*}
\end{thm}

\begin{proof}
The proof is almost the same (making transpose of the partitions) as in the case of $ \UU(2n)/\OO(2n) $.  
\end{proof}

\bigskip
\bigskip


\printbibliography 

@article{Avdeev.Petukhov.2020,
 author = {Avdeev, Roman and Petukhov, Alexey},
 title = {Branching rules related to spherical actions on flag varieties},
 fjournal = {Algebras and Representation Theory},
 journal = {Algebr. Represent. Theory},
 issn = {1386-923X},
 volume = {23},
 number = {3},
 pages = {541--581},
 year = {2020},
 language = {English},
 doi = {10.1007/s10468-019-09857-9},
 zbMATH = {7210918},
 Zbl = {1454.20089}
}

@article{Benson.Ratcliff.1996,
 author = {Benson, Chal and Ratcliff, Gail},
 title = {A classification of multiplicity free actions},
 fjournal = {Journal of Algebra},
 journal = {J. Algebra},
 issn = {0021-8693},
 volume = {181},
 number = {1},
 pages = {152--186},
 year = {1996},
 language = {English},
 doi = {10.1006/jabr.1996.0113},
 zbMATH = {883753},
 Zbl = {0869.14021}
}

@article{Borel.1953,
  title={Sur la cohomologie des espaces fibr{\'e}s principaux et des espaces homogenes de groupes de Lie compacts},
  author={Borel, Armand},
  journal={Annals of Mathematics},
  volume={57},
  number={1},
  pages={115--207},
  year={1953}
}

@misc{Brion.1989,
    author = {Brion, Michel},
    title = {Spherical varieties: {An} introduction},
    year = {1989},
    language = {English},
    howpublished = {Topological methods in algebraic transformation groups, {Proc}. {Conf}., {New} {Brunswick}/{NJ} ({USA}) 1988, {Prog}. {Math}. 80, 11-26.},
    zbMATH = {4193898},
    Zbl = {0724.14034}
}

@article{CGKP.2025,
title = {Clifford algebra analogue of Cartan's theorem for symmetric pairs},
author = {Kieran Calvert and Karmen Grizelj and Andrey Krutov and Pavle Pandžić},
journal = {Advances in Mathematics},
volume = {503B},
eid = {111196},
year = {2026},
issn = {0001-8708},
doi = {https://doi.org/10.1016/j.aim.2026.111196},
url = {https://www.sciencedirect.com/science/article/pii/S0001870826004172},
label = "CGKP",
}

@misc{CNP.arxiv2023,
 author = {Kieran Calvert and Kyo Nishiyama and Pavle Pand{\v{z}}i{\'c}},
 title = {Clifford algebras, symmetric spaces and cohomology rings of {Grassmannians}},
 year = {2023},
 url = {https://arxiv.org/abs/2310.04839},
 archivePrefix = {arXiv},
 eprint = {2310.04839},
 arXiv = {arXiv:2310.04839},
 primaryClass={math.RT},
label = "CNP",
}

@incollection{Fresse.Nishiyama.Overview,
 author = {Fresse, Lucas and Nishiyama, Kyo},
 title = {Overview on the theory of double flag varieties for symmetric pairs},
 booktitle = {Representations and characters. Revisiting the works of Harish-Chandra and Andr\'e Weil. Selected papers based on the presentations at the conference, IMS, Singapore, July 1--15, 2022},
 isbn = {978-981-98-2543-1; 978-981-98-2545-5},
 pages = {137--224},
 year = {2026},
 publisher = {Singapore: World Scientific},
 language = {English},
 doi = {10.1142/9789819825448_0005},
 zbMATH = {8200661}
}

@article{HNOO.2013,
  author = {Xuhua He and Kyo Nishiyama and Hiroyuki Ochiai and Yoshiki Oshima},
  title = {On orbits in double flag varieties for symmetric pairs}, 
   JOURNAL = {Transform. Groups},
  FJOURNAL = {Transformation Groups},
    VOLUME = {18},
      YEAR = {2013},
     PAGES = {1091--1136},
label = "HNOO",
}

@book{CCC.III.1976,
 author = {Greub, Werner and Halperin, Stephen and Vanstone, Ray},
 title = {Connections, curvature, and cohomology. {Vol}. {III}: {Cohomology} of principal bundles and homogeneous spaces},
 fseries = {Pure and Applied Mathematics (Academic Press)},
 series = {Pure Appl. Math., Academic Press},
 issn = {0079-8169},
 volume = {47},
 year = {1976},
 publisher = {Academic Press, New York, NY},
 language = {English},
 zbMATH = {3579627},
 Zbl = {0372.57001}
}

@incollection{Howe.SchurLecture.1995,
 author = {Howe, Roger},
 title = {Perspectives on invariant theory: {Schur} duality, multiplicity-free actions and beyond},
 booktitle = {The Schur lectures (1992)},
 pages = {1--182},
 year = {1995},
 publisher = {Ramat-Gan: Bar-Ilan University; Providence, RI: American Mathematical Society (Distrib.)},
 language = {English},
 zbMATH = {796484},
 Zbl = {0844.20027}
}

@book{HPbook,
 author = {Huang, Jing-Song and Pand{\v{z}}i{\'c}, Pavle},
 title = {Dirac operators in representation theory},
 isbn = {0-8176-3218-2},
 year = {2006},
 publisher = {Basel: Birkh{\"a}user},
 language = {English},
 zbMATH = {5044774},
 Zbl = {1103.22008}
}

@article{HPRarXiv,
      title={Dirac operators and Lie algebra cohomology}, 
      author={Jing-Song Huang and Pavle Pandžić and David Renard},
      year={2005},
      eprint={math/0503582},
      archivePrefix={arXiv},
      primaryClass={math.RT},
      url={https://arxiv.org/abs/math/0503582}, 
}

@article{Kac.1980,
 author = {Kac, Victor G.},
 title = {Some remarks on nilpotent orbits},
 fjournal = {Journal of Algebra},
 journal = {J. Algebra},
 issn = {0021-8693},
 volume = {64},
 pages = {190--213},
 year = {1980},
 language = {English},
 doi = {10.1016/0021-8693(80)90141-6},
 zbMATH = {3670645},
 Zbl = {0431.17007}
}

@article{Kobayashi.Nagano.1964,
  author  = {Kobayashi, Shoshichi and Nagano, Tadashi},
  title   = {On filtered {L}ie algebras and geometric structures. {I}},
  journal = {Journal of Mathematics and Mechanics},
  volume  = {13},
  year    = {1964},
  pages   = {875--907},
  mrnumber = {0168704},
  url     = {https://www.jstor.org/stable/24901239}
}

@article{Kobayashi.Nagano.1965,
  author  = {Kobayashi, Shoshichi and Nagano, Tadashi},
  title   = {On filtered {L}ie algebras and geometric structures. {II}},
  journal = {Journal of Mathematics and Mechanics},
  volume  = {14},
  year    = {1965},
  pages   = {513--521},
  mrnumber = {0185042},
  url     = {https://www.jstor.org/stable/24901294}
}

@article{Knapp.1975,
 author = {Knapp, Anthony W.},
 title = {Weyl group of a cuspidal parabolic},
 fjournal = {Annales Scientifiques de l'{\'E}cole Normale Sup{\'e}rieure. Quatri{\`e}me S{\'e}rie},
 journal = {Ann. Sci. {\'E}c. Norm. Sup{\'e}r. (4)},
 issn = {0012-9593},
 volume = {8},
 pages = {275--294},
 year = {1975},
 language = {English},
 doi = {10.24033/asens.1288},
 url = {https://eudml.org/doc/81957},
 zbMATH = {3477435},
 Zbl = {0305.22010}
}

@incollection{Knop.MFA.1998,
 author = {Knop, Friedrich},
 title = {Some remarks on multiplicity free spaces},
 booktitle = {Representation theories and algebraic geometry. Proceedings of the NATO Advanced Study Institute, Montreal, Canada, July 28--August 8, 1997},
 isbn = {0-7923-5193-2},
 pages = {301--317},
 year = {1998},
 publisher = {Dordrecht: Kluwer Academic Publishers},
 language = {English},
 zbMATH = {1215406},
 Zbl = {0915.20021}
}

@article {Kostant.I.1961,
    AUTHOR = {Kostant, Bertram},
     TITLE = {Lie algebra cohomology and the generalized {B}orel-{W}eil
              theorem},
   JOURNAL = {Ann. of Math. (2)},
  FJOURNAL = {Annals of Mathematics. Second Series},
    VOLUME = {74},
      YEAR = {1961},
     PAGES = {329--387},
      ISSN = {0003-486X},
   MRCLASS = {22.60 (22.70)},
  MRNUMBER = {142696},
MRREVIEWER = {A.\ Borel},
       DOI = {10.2307/1970237},
       URL = {https://doi.org/10.2307/1970237},
}

@article{Kostant.LMP.2000,
 author = {Kostant, Bertram},
 title = {A generalization of the {Bott}-{Weil} theorem and {Euler} number multiplets of representations},
 fjournal = {Letters in Mathematical Physics},
 journal = {Lett. Math. Phys.},
 issn = {0377-9017},
 volume = {52},
 number = {1},
 pages = {61--78},
 year = {2000},
 language = {English},
 doi = {10.1023/A:1007653819322},
 zbMATH = {1545427},
 Zbl = {0960.22011}
}

@article{Leahy.1998,
 author = {Leahy, Andrew S.},
 title = {A classification of multiplicity free representations},
 fjournal = {Journal of Lie Theory},
 journal = {J. Lie Theory},
 issn = {0949-5932},
 volume = {8},
 number = {2},
 pages = {367--391},
 year = {1998},
 language = {English},
 url = {https://eudml.org/doc/229504},
 zbMATH = {1211709},
 Zbl = {0910.22015}
}

@book{Macdonald.1995,
  author    = {Macdonald, Ian G.},
  title     = {Symmetric Functions and {Hall} Polynomials},
  edition   = {2},
  series    = {Oxford Mathematical Monographs},
  publisher = {The Clarendon Press, Oxford University Press},
  address   = {New York},
  year      = {1995},
  pages     = {x+475},
  isbn      = {0-19-853489-2},
  doi       = {10.1093/oso/9780198534891.001.0001},
  note      = {With contributions by A. Zelevinsky},
  mrnumber  = {1354144}
}

@article{Nagano.1965,
  author  = {Nagano, Tadashi},
  title   = {Transformation groups on compact symmetric spaces},
  journal = {Trans. Amer. Math. Soc.},
  volume  = {118},
  year    = {1965},
  pages   = {428--453},
  doi     = {10.2307/1993971},
  mrnumber = {0182937}
}

@book{Onishchik.1994,
 author = {Onishchik, Arkadi L.},
 title = {Topology of transitive transformation groups},
 isbn = {3-335-00355-1},
 year = {1994},
 publisher = {Leipzig: Johann Ambrosius Barth},
 language = {English},
 zbMATH = {523999},
 Zbl = {0796.57001}
}

@article{Pecher.2012,
 author = {Pecher, Tobias},
 title = {Classification of skew multiplicity-free modules.},
 fjournal = {Transformation Groups},
 journal = {Transform. Groups},
 issn = {1083-4362},
 volume = {17},
 number = {1},
 pages = {233--257},
 year = {2012},
 language = {English},
 doi = {10.1007/s00031-011-9171-4},
 zbMATH = {6050318},
 Zbl = {1257.20048}
}

@article{RRS.1992,
 author = {Richardson, Roger and R{\"o}hrle, Gerhard and Steinberg, Robert},
 title = {Parabolic subgroups with {Abelian} unipotent radical},
 fjournal = {Inventiones Mathematicae},
 journal = {Invent. Math.},
 issn = {0020-9910},
 volume = {110},
 number = {3},
 pages = {649--671},
 year = {1992},
 language = {English},
 doi = {10.1007/BF01231348},
 url = {https://eudml.org/doc/144066},
 zbMATH = {224071},
 Zbl = {0786.20029}
}

@book {Stanley.vol1.2012,
    AUTHOR = {Stanley, Richard P.},
     TITLE = {Enumerative combinatorics. {V}olume 1},
    SERIES = {Cambridge Studies in Advanced Mathematics},
    VOLUME = {49},
   EDITION = {Second},
 PUBLISHER = {Cambridge University Press, Cambridge},
      YEAR = {2012},
     PAGES = {xiv+626},
      ISBN = {978-1-107-60262-5},
   MRCLASS = {05-02 (05A15 06-02)},
  MRNUMBER = {2868112},
}

@article {Takeuchi.1962,
    AUTHOR = {Takeuchi, Masaru},
     TITLE = {On {P}ontrjagin classes of compact symmetric spaces},
   JOURNAL = {J. Fac. Sci. Univ. Tokyo Sect. I},
  FJOURNAL = {Journal of the Faculty of Science. University of Tokyo.
              Section I},
    VOLUME = {9},
      YEAR = {1962},
     PAGES = {313--328},
      ISSN = {0368-2269},
   MRCLASS = {57.32 (57.45)},
  MRNUMBER = {145009},
MRREVIEWER = {S.\ Murakami},
}

@article {Cart50,
    AUTHOR = {Cartan, Henri},
     TITLE = {La Trasngression dans un group de {L}ie et dans un fibr\'{e} principal},
   JOURNAL = {Collogue de topologie (espaces fibr\'{e}s)},
  FJOURNAL = {},
    VOLUME = {},
      YEAR = {1950},
    NUMBER = {},
     PAGES = {73--81},
      ISSN = {},
   MRCLASS = {},
  MRNUMBER = {},
       DOI = {},
       URL = {},
}

@article{Casian.Kodama.2013,
      title={On the cohomology of real {G}rassmann manifolds}, 
      author={Luis Casian and Yuji Kodama},
      year={2013},
 archivePrefix = {arXiv},
 eprint = {1309.5520},
 arXiv = {arXiv:1309.5520},
      primaryClass={math.AG}
}

@article {Goette.MZ.1999,
    AUTHOR = {Goette, Sebastian},
     TITLE = {Equivariant {$\eta$}-invariants on homogeneous spaces},
   JOURNAL = {Math. Z.},
  FJOURNAL = {Mathematische Zeitschrift},
    VOLUME = {232},
      YEAR = {1999},
    NUMBER = {1},
     PAGES = {1--42},
      ISSN = {0025-5874,1432-1823},
   MRCLASS = {58J28 (58J20)},
  MRNUMBER = {1714278},
MRREVIEWER = {Kai\ K\"ohler},
       DOI = {10.1007/PL00004757},
       URL = {https://doi.org/10.1007/PL00004757},
}

@book {beyond,
    AUTHOR = {Knapp, Anthony W.},
     TITLE = {Lie groups beyond an introduction},
    SERIES = {Progress in mathematics },
    VOLUME = {140},
 PUBLISHER = {Birkh{\"a}user},
      YEAR = {1996},
     PAGES = {},
      ISBN = {0817639268},
   MRCLASS = {},
  MRNUMBER = {},
MRREVIEWER = {},
       DOI = {},
       URL = {https://doi.org/},
}

@article {Kostant.cubic.1999,
    AUTHOR = {Kostant, Bertram},
     TITLE = {A cubic {D}irac operator and the emergence of {E}uler number
              multiplets of representations for equal rank subgroups},
   JOURNAL = {Duke Math. J.},
  FJOURNAL = {Duke Mathematical Journal},
    VOLUME = {100},
      YEAR = {1999},
    NUMBER = {3},
     PAGES = {447--501},
      ISSN = {0012-7094,1547-7398},
   MRCLASS = {22E46 (19L64 20C35 58J05)},
  MRNUMBER = {1719734},
MRREVIEWER = {A.\ L.\ Onishchik},
       DOI = {10.1215/S0012-7094-99-10016-0},
       URL = {https://doi.org/10.1215/S0012-7094-99-10016-0},
}

@article{TK1968,
title={Minimal imbeddings of R-spaces},
author={Takeuchi, Masaru and  Kobayashi, Shoshichi}, 
journal={J. Differential Geometry},  
volume={2},  
year={1968}, 
pages={203--215},
}
\printindex
\end{document}